\documentclass[12pt]{article}

\usepackage[margin=1in]{geometry}
\usepackage{amsmath,amsthm,amssymb}
\usepackage{enumerate}
\usepackage{xfrac}
\usepackage{accents} 
\usepackage[colorlinks=true,urlcolor=magenta,citecolor=magenta,linkcolor=magenta]{hyperref}
    \AtBeginDocument{}
\usepackage{tikz} 
\usetikzlibrary{arrows}
\usetikzlibrary{patterns}
\usepackage{pgfplots} 
\usetikzlibrary{decorations.markings} 
\usetikzlibrary{3d,calc} 
\usepackage{pifont}
\usepackage{mathtools}
\usepackage[T1]{fontenc}

\usepackage{fancyhdr}
\newcommand{\N}{\mathbb{N}}
\newcommand{\Z}{\mathbb{Z}}
\newcommand{\Q}{\mathbb{Q}}

\newcommand{\C}{\mathbb{C}}
\newcommand{\F}{\mathbb{F}}
\renewcommand{\L}{\mathcal{L}}

\renewcommand{\P}{\mathbb{P}}
\newcommand{\M}{\mathcal{M}}
\let\SS\S
\renewcommand{\S}{\mathcal{S}}

\newcommand{\G}{\mathbb{G}}
\newcommand{\W}{\mathbb{W}}
\newcommand{\X}{\mathcal{X}}
\newcommand{\Y}{\mathcal{Y}}

\DeclareFontFamily{U}{wncy}{}
    \DeclareFontShape{U}{wncy}{m}{n}{<->wncyr10}{}
    \DeclareSymbolFont{mcy}{U}{wncy}{m}{n}
    \DeclareMathSymbol{\Sha}{\mathord}{mcy}{"58} 

\newcommand{\rig}{\ensuremath{\operatorname{rig}}}
\newcommand{\SL}{\ensuremath{\operatorname{SL}}}
\newcommand{\GL}{\ensuremath{\operatorname{GL}}}
\newcommand{\PSL}{\ensuremath{\operatorname{PSL}}}
\newcommand{\Sp}{\ensuremath{\operatorname{Sp}}}
\newcommand{\ns}{\ensuremath{\operatorname{ns}}}
\newcommand{\spl}{\ensuremath{\operatorname{sp}}}

\renewcommand{\char}{\ensuremath{\operatorname{char}}}
\newcommand{\ch}{\ensuremath{\operatorname{ch}}}
\newcommand{\Spec}{\ensuremath{\operatorname{Spec}}}

\newcommand{\im}{\ensuremath{\operatorname{im}}}

\newcommand{\Aut}{\ensuremath{\operatorname{Aut}}}

\newcommand{\la}{\langle}
\newcommand{\ra}{\rangle}
\newcommand{\orb}{\mathcal{O}}

\newcommand{\Ext}{\ensuremath{\operatorname{Ext}}}

\newcommand{\tr}{\ensuremath{\operatorname{tr}}}
\newcommand{\Tr}{\ensuremath{\operatorname{Tr}}}

\newcommand{\Gal}{\ensuremath{\operatorname{Gal}}}

\newcommand{\A}{\mathbb{A}}

\newcommand{\Pic}{\ensuremath{\operatorname{Pic}}}
\renewcommand{\div}{\ensuremath{\operatorname{div}}}

\newcommand{\Proj}{\ensuremath{\operatorname{Proj}}}

\newcommand{\an}{\ensuremath{\operatorname{an}}}
\newcommand{\Ig}{\ensuremath{\operatorname{Ig}}}

\newcommand{\h}{\frak{h}}

\renewcommand{\epsilon}{\varepsilon}

\newcommand{\ord}{\ensuremath{\operatorname{ord}}}
\newcommand{\legen}[2]{\left (\frac{#1}{#2}\right )}

\newcommand{\Tate}{\ensuremath{\operatorname{Tate}}}
\newcommand{\Frob}{\ensuremath{\operatorname{Frob}}}

\newtheorem{thm}{Theorem}[section]
\newtheorem{prop}[thm]{Proposition}

\newtheorem{lem}[thm]{Lemma}
\newtheorem*{defn}{Definition}
  \let\olddefn\defn
  \renewcommand{\defn}{\olddefn\normalfont}
\newtheorem{cor}[thm]{Corollary}

\newtheorem{rem}[thm]{Remark}
  \let\oldrem\rem
  \renewcommand{\rem}{\oldrem\normalfont}
\newtheorem{question}[thm]{Question}
\newtheorem{problem}[thm]{Problem}
\newtheorem{ex}[thm]{Example}
  \let\oldex\ex
  \renewcommand{\ex}{\oldex\normalfont}

\newcommand{\ds}{\displaystyle}

\newcommand{\magma}{\texttt{Magma}}

\newcommand{\repolink}[1]{\href{https://github.com/dmzb/kzb-modular-forms-mod-p/blob/main/#1}{\texttt{#1}}}

\begin{document}



\rhead{}

\title{Wild Stacky Curves and Rings of Mod $p$ Modular Forms}
\author{Andrew Kobin and David Zureick-Brown}
\date{}

\maketitle

\begin{abstract}
    We extend work of Voight and the second author to compute the log canonical ring of a wild stacky curve over a field of characteristic $p > 0$, which allows us to compute rings of mod $p$ modular forms of level $\Gamma_{0}(N)$. Our approach also reveals that in characteristics $2$ and $3$, there are infinitely many levels $N$ for which there are weight $2$ modular forms of level $\Gamma_{0}(N)$ that do not lift to characteristic $0$. 
\end{abstract}

\section{Introduction}
The {\it canonical ring} of an algebraic curve $X$, defined as 
$$
R(X) = \bigoplus_{k = 0}^{\infty} H^{0}(X,\Omega_{X}^{\otimes k})
$$
where $\Omega_{X}$ is the canonical line bundle on $X$, is a useful and well-studied algebraic invariant of $X$. Its graded pieces encode information about embeddings of $X$ into projective space; for example, when $X$ is hyperbolic (i.e.~its genus $g = g(X)$ is at least $2$) and not hyperelliptic, $\Omega_{X}$ determines an embedding $X\hookrightarrow\P^{g - 1}$ whose image is isomorphic to $\Proj R(X)$. In general, $R(X)$ captures essential geometric features of $X$ and its projective models, deformations, etc.

In \cite{vzb}, Voight and the second author extend the theory of $R(X)$ to any tame log stacky curve $(\X,\Delta)$ and provide a presentation of the corresponding {\it log canonical ring} $R(\X,\Delta)$ in terms of explicit generators and relations. Our first main result extends their work to wild stacky curves. 

For a (possibly wild) log stacky curve $(\X,\Delta)$, let $X$ be the coarse space of $\X$, with coarse moduli map $\pi \colon \X\to X$, and set $g = g(X)$. Label the finitely many stacky points of $\X$ by $P_{1},\ldots,P_{r}$ and let $c_{i}$ denote the coefficient of $\pi(P_{i})$ in the $\Q$-divisor $\pi_{*}K_{\X}$ and set $c = \sum_{i = 1}^{r} \lfloor c_{i}\rfloor$. Put $\delta = \deg(\Delta)$. We call the tuple $(g;c_{1},\ldots,c_{r};\delta)$ the {\it refined signature} of $(\X,\Delta)$. 

\begin{thm}[{Theorem~\ref{thm:wildVZB}}]
\label{thm:wildVZB-intro}
For a (possibly wild) separably rooted log stacky curve $(\X,\Delta)$ with refined signature $(g;c_{1},\ldots,c_{r};\delta)$, the log canonical ring $R(\X,\Delta)$ admits a presentation with generators in degrees $\leq 3e$ and relations in degrees $\leq 6e$, where $e$ is the largest denominator of $c_{1},\ldots,c_{r}$ when they are written in lowest terms. Moreover, when $g + c + \delta\geq 2$, there is a presentation with generators in degrees $\leq \max(3,e)$ and relations in degrees $\leq 2\max(3,e)$. 
\end{thm}

This generalizes \cite[Thm.~1.4.1]{vzb} and extends \cite[Thm.~1.4]{cerchiaFO:section-rings-genus-1} to higher genus curves in the case of a log canonical divisor. A key principle in the proof of Theorem~\ref{thm:wildVZB-intro} is that ``wild ramification forces generators into lower degrees''. This will be made precise in Section~\ref{sec:rustom} and illustrated in the sequence of examples in \SS\ref{sec:rustom-examples}. 

\subsection{Rustom's conjecture in characteristic $p$}

One of the main applications of \cite[Thm.~1.4.1]{vzb} is to prove a conjecture of Rustom \cite{rus} concerning generators and relations in graded rings of modular forms, namely that for any $N\geq 1$, the graded ring 
$$
M_{\bullet}\left (N;\Z\left [\tfrac{1}{6N}\right ]\right ) = \bigoplus_{k = 0}^{\infty} M_{k}\left (N;\Z\left [\tfrac{1}{6N}\right ]\right )
$$
of modular forms of level $\Gamma_{0}(N)$ with coefficients in $\Z\left [\frac{1}{6N}\right ]$ is generated in weights $\leq 6$ with relations in weights $\leq 12$. It is crucial that one inverts $6N$ rather than just $N$, since the conjecture is actually false in the latter case (see Example~\ref{ex:rustomfalse}). 

The connection between Rustom's conjecture and log canonical rings of stacky curves is realized by the Kodaira--Spencer isomorphism \cite[Lem.~6.2.3]{vzb} 
\begin{equation}\label{eq:KS}
M_{k}(N;\C) \xrightarrow{\;\sim\;} H^{0}(\X_{0}(N)^{\rig},\Omega_{\X_{0}(N)^{\rig}}(\Delta)^{\otimes k/2})
\end{equation}
between the space of weight $k$, level $\Gamma_{0}(N)$ classical modular forms and an appropriate space of sections of the log canonical bundle on the stacky curve $\X_{0}(N)^{\rig}$, the rigidification of the moduli stack of complex elliptic curves with level $\Gamma_{0}(N)$-structure, equipped with $\Delta$, the divisor of cusps. Following Katz \cite{kat}, one can use the right side of (\ref{eq:KS}) as a \emph{definition} of modular forms over an arbitrary ring $R$, giving a geometric interpretation of the graded ring $M_{\bullet}\left (N;\Z\left [\frac{1}{6N}\right ]\right )$ in Rustom's conjecture (see Section~\ref{sec:geoMF}). 

From this geometric perspective, in \cite{vzb} it was necessary to invert $6$ because in characteristics $2$ and $3$, the stacky curve $\X_{0}(N)$ is often \emph{wildly ramified} and therefore the main theorem in [{\it loc.~cit.}] does not apply. Indeed, Rustom's conjecture as originally stated is false for $\Z\left [\frac{1}{N}\right ]$-coefficients, but only slightly. Using Theorem~\ref{thm:wildVZB-intro}, we prove the following refinement of the conjecture. 

\begin{thm}[{Theorem~\ref{thm:wildrustom}}]
\label{thm:wildrustom-intro}
For $N\geq 1$, the graded ring of modular forms $M_{\bullet}\left (N;\Z\left [\frac{1}{N}\right ]\right )$ has a presentation with generators and relations in weights $\leq 12$. Moreover, for $N > 1$, the generators appear in weights $\leq 6$ with relations in weights $\leq 12$, as in the case of Rustom's original conjecture. 
\end{thm}

This answers a question of the first author \cite[Question 1]{kob}. 
A similar statement should be true for any modular curve $\X_{H}$ with $H\leq\SL_{2}(\Z)$ a subgroup of finite index. See Section~\ref{sec:data} for one example.

\subsection{Serre's modularity conjecture and ethereal forms}

Our methods for proving Theorem~\ref{thm:wildVZB-intro} reveal another fascinating aspect of rings of mod $p$ modular forms, by way of Serre's modularity conjectures. Serre's conjectures concern the modularity of certain types of Galois representations in the form of a correspondence between Galois representations and modular forms. More concretely, work of Deligne and Serre \cite{deligne-serre} constructs, for every normalized cuspidal eigenform $f\in S_{k}(N;\overline{\F}_{p})$, a continuous, odd, irreducible Galois representation $\rho_{f} \colon G_{\Q}\rightarrow \GL_{2}(\overline{\F}_{p})$. The strong form of {\it Serre's weight conjecture} \cite{ser} asserts the converse: for every such $\rho \colon G_{\Q}\rightarrow \GL_{2}(\overline{\F}_{p})$, there are a well-defined weight $k$ and level $N$ and a normalized cuspidal eigenform $f\in S_{k}(N;\overline{\F}_{p})$ whose associated $\rho_{f}$ is $\rho$. The conjecture was ultimately proven by Khare--Wintenberger \cite{kw1,kw2} and Kisin \cite{kis}, while the generalization to number fields remains an open problem. 

One of the difficulties on the modular form side of the story is that mod $p$ modular forms may differ from their classical counterparts. As above, let $M_{k}(N;R)$ denote the space of weight $k$, level $N$ modular forms with coefficients in a ring $R$, in the sense of Katz \cite{kat}; see \SS\ref{sec:geoMF} for the rigorous definition. For any prime $p$, there is a reduction mod $p$ map 
\[
M_{k}\left (N;\Z\left [\tfrac{1}{N}\right ]\right ) \longrightarrow M_{k}(N;\F_{p})
\]
which may fail to be surjective. That is, \emph{not every mod $p$ modular form lifts to a classical modular form}. Classical here means ``after tensoring with $\C$'', in which case one obtains a modular form on the upper half plane. 

By Serre's modularity conjectures, the so-called {\it ethereal modular forms}, i.e.~those lying outside the image of the reduction map for some $p$, give rise to Galois representations $\rho \colon G_{\Q}\rightarrow\GL_{2}(\F_{p})$ which are ``non-classical'' in that they do not lift to a modular representation $G_{\Q}\rightarrow\GL_{2}(\overline{\Z}_{p})$ of the same weight and level as $\rho$. Additionally, as pointed out in \cite{buz} and \cite{sch}, the appearance of Galois representations coming from ethereal modular forms has an influence on the arithmetic statistics of number fields -- in particular, their images in $\GL_{2}(\F_{p})$ are unusually large, as discussed in \cite{buz}. 

In this article, we give a geometric explanation for the existence of some ethereal mod $p$ modular forms by way of the stack structure of moduli problems of elliptic curves with level structure, i.e.~stacky modular curves. This stacky approach also allows us to compute the ring of mod $p$ modular forms with level structure, extending the methods in \cite{vzb} to characteristic $p$. 

Our key tool for discovering and computing ethereal modular forms is Theorem~\ref{thm:wildrustom-intro}. Indeed, one of the key predictors of ethereal modular forms mod $p$ is the presence of wild ramification in the moduli stacks $\X_{0}(N)$. Already for the moduli stack $\X(1)$ of elliptic curves, wild ramification occurs in characteristics $2$ and $3$, leading to a more exotic stacky structure in these characteristics (Proposition~\ref{prop:P46mod23}) which propagates to higher levels through the tower of modular curves (\SS\ref{sec:level}). 

The phenomenon of ethereal modular forms for the subgroups $\Gamma_{1}(N)$ was first noticed by Mestre \cite{mes} and later studied extensively in \cite{buz,sch} but, to our knowledge, there does not exist a thorough treatment of ethereal forms for $\Gamma_{0}(N)$ in the literature. The present work can be regarded as a first step in the direction of a more comprehensive account of ethereal modular forms and their Galois representations. See Section~\ref{sec:future} for further discussion and suggested directions of inquiry. 

\subsection{Computing ethereal modular forms}
\label{subsec:algorithm}

Here's a brief outline of our approach to identifying and computing ethereal forms in $M_{\bullet}(\Gamma;\F_{p})$. In the present article, $\Gamma$ is always $\Gamma_{0}(N)$, but can be an arbitrary Fuchsian group in principle (for one example, see Section~\ref{sec:data}). The strategy can be divided into two parts: 
\begin{enumerate}[(I)]
    \item Compute a presentation for the log canonical ring $R(\X,\Delta)$, where $\X$ is the modular curve satisfying $M_{k}(\Gamma;\F_{p}) \cong H^{0}(\X,\Omega_{\X}(\Delta)^{k})$; for $\Gamma = \Gamma_{0}(N)$, this is the rigidification $\X_{0}(N)^{\rig}$ (see \cite[Rmk.~5.6.8 and Lem.~6.2.3]{vzb}). 
    \item Use a known basis of (non-ethereal) modular forms and linear algebra to produce $q$-expansions of ethereal generators in low weights. 
\end{enumerate}

In the first step, one must determine the stack structure of $\X$, namely the number of stacky points and their automorphism groups. For the curves $\X_{0}(N)^{\rig}$, this is done in Sections~\ref{sec:prelim} and~\ref{sec:X0N-stackyness} and is summarized by the following theorem. 

\begin{thm}[{Theorem~\ref{thm:P46modp}, Corollaries~\ref{cor:stackycount} and~\ref{cor:stackycount23}}]
\label{thm:X0N-intro}
Let $N\geq 1$ and define 
\begin{align*}
    \epsilon_{2}(N) &= \begin{cases}
        \prod_{\text{odd primes } \ell\mid N} \left (1 + \left (\frac{-1}{\ell}\right )\right ), &\text{if } 4\nmid N\\
        0, &\text{if } 4\mid N
    \end{cases}\\
    \text{and}\quad \epsilon_{3}(N) &= \begin{cases}
        \prod_{\text{primes } 3\not = \ell\mid N} \left (1 + \left (\frac{-3}{\ell}\right )\right ), &\text{if } 9\nmid N\\
        0, &\text{if } 9\mid N. 
    \end{cases}
\end{align*}
Over any algebraically closed field $k$ of characteristic not dividing $N$, $\X_{0}(N)^{\rig}$ is a stacky curve with coarse space $X_{0}(N)$, the usual modular curve of level $\Gamma_{0}(N)$, whose stacky locus is characterized as follows: 
\begin{enumerate}[\quad (1)]
    \item If $N = 1$, then either 
    \begin{enumerate}[(a)]
        \item $\char k\not = 2,3$ and $\X_{0}(1)^{\rig} = \X(1)^{\rig}$ has a stacky $\mu_{2}$-point over $j = 1728$ and a stacky $\mu_{3}$-point over $j = 0$; or
        \item $\char k = 2$ and $\X_{0}(1)^{\rig}$ has a single wild stacky $\Z/3\Z\rtimes(\Z/2\Z\times\Z/2\Z)$-point over $j = 0$; or
        \item $\char k = 3$ and $\X_{0}(1)^{\rig}$ has a single wild stacky $S_{3}$-point over $j = 0$. 
    \end{enumerate}
    \item Otherwise, $N > 1$ and the stacky locus is characterized by:
    \begin{enumerate}[(a)]
        \item If $\char k = 0$ or $\char k > 3$, then $\X_{0}(N)^{\rig}$ has $\epsilon_{2}(N)$ stacky $\mu_{2}$-points over $j = 1728$ and $\epsilon_{3}(N)$ stacky $\mu_{3}$-points over $j = 0$. 
        \item If $\char k = 2$, then $\X_{0}(N)^{\rig}$ has $\epsilon_{2}(N)/2$ stacky $\Z/2\Z$-points and $\epsilon_{3}(N)$ stacky $\mu_{3}$-points over $j = 0$. 
        \item If $\char k = 3$, then $\X_{0}(N)^{\rig}$ has $\epsilon_{2}(N)$ stacky $\mu_{2}$-points and $\epsilon_{3}(N)/2$ stacky $\Z/3\Z$-points over $j = 0$. 
    \end{enumerate}
\end{enumerate}
\end{thm}

To compute the log canonical ring of $(\X,\operatorname{cusps})$ also requires knowing the ramification jump(s) at each wild stacky point. This can be computed indirectly using an appropriate \'{e}tale cover $Y\rightarrow\X$, where $Y$ is a representable curve, usually $X_{1}(N)$ for some $N\geq 5$. With the stacky structure determined, one may then use the wild stacky Riemann--Hurwitz formula \cite[Prop.~7.1]{kob} to compute a canonical divisor for $\X$. Finally, the main results in \cite{ODorn} and \cite{vzb} help in computing the log canonical ring of $\X$. 

Completing task (I) already allows us to characterize the existence of ethereal modular forms for $\X_{0}(N)^{\rig}$, which is summarized by our next main theorem. 

\begin{thm}[{Theorem~\ref{thm:levels}}]
\label{thm:levels-intro}
Fix a prime $p$ and $N\geq 1$ not divisible by $p$. Then the ring $M_{\bullet}(N;\F_{p})$ of mod $p$ modular forms of level $\Gamma_{0}(N)$ contains an ethereal form if and only if one of the following is true: 
\begin{enumerate}[\quad (1)]
    \item $p = 2$ and $N$ is a product of primes $\ell\equiv 1\pmod{4}$. 
    \item $p = 3$ and $N$ is a product of primes $\ell\equiv 1\pmod{3}$. 
\end{enumerate}
\end{thm}

In fact, the enumeration of stacky points in Theorem~\ref{thm:X0N-intro} allows us to count the number of ethereal modular forms in weight $2$ (see Theorem~\ref{thm:levelforms}). 

Task (II) is a purely computational exercise. Given a prime $p$, a subgroup $\Gamma\subseteq\SL_{2}(\Z)$, a weight $k$ and a precision $t > 0$, we would like to produce a finite set $\{f_{1},\ldots,f_{d}\}\subset\F_{p}[[q]]/(q^{t + 1})$ consisting of truncated $q$-expansions of a basis for the space of modular forms $M_{k}(\Gamma;\F_{p})$, where $\dim M_{k}(\Gamma;\F_{p}) = d$. For certain $\Gamma$ and $p$, one must account for ethereal modular forms, as in \cite{sch} and Theorem~\ref{thm:levels-intro}. We outline the strategy for $\Gamma = \Gamma_{0}(N)$, where $N\geq 1$ is not divisible by $p$: 
\begin{enumerate}[\quad (i)]
    \item Compute a basis $\{g_{1},\ldots,g_{r}\}$ for the space $M_{k}\left (N;\Z\left [\frac{1}{N}\right ]\right )$ of characteristic $0$ modular forms, e.g.~using Magma or Sage, and reduce the forms mod $p$, expressing them as truncated $q$-expansions $\bar{g}_{i} = \sum_{n = 0}^{t} a_{n}(g_{i})q^{n} + O(q^{t + 1})\in\F_{p}[[q]]$. 
    \item Use the results of task (I), namely a presentation of $M_{\bullet}(N;\F_{p})$ coming from Theorem~\ref{thm:wildVZB-intro}, to determine a basis for $M_{k}(N;\F_{p})$ consisting of monomials of low-weight forms; by Theorem~\ref{thm:wildrustom-intro}, these can be found in weights $\leq 6$ when $N > 1$. 
    \item Some of the low-weight forms will be ethereal (the number of ethereal basis elements can be determined by dimension formulas, as in Theorem~\ref{thm:levelforms}). Use linear algebra to deduce relations among the monomials and the $\bar{g}_{i}$, culminating in a maximal linearly independent set $\{\bar{g}_{1},\ldots,\bar{g}_{r},f_{r + 1},\ldots,f_{d}\}$, where $f_{r + 1},\ldots,f_{d}$ are ethereal monomials. 
\end{enumerate}
In practice, only finitely many terms of the $q$-expansions $\bar{g}_{1},\ldots,\bar{g}_{r}$ are needed to discover $q$-expansions for all ethereal generators in low weights; see Examples~\ref{ex:ethereallevel5mod2} -- \ref{ex:ethereallevel91mod3} for concrete implementations of this strategy. 

%

\subsection{Organization}

The paper is organized as follows. In Section~\ref{sec:prelim}, we gather some important preliminaries on modular forms (\SS\ref{sec:classicMF}), geometric modular forms (\SS\ref{sec:geoMF}), mod $p$ modular forms (\SS\ref{sec:MFmodp} - \ref{sec:mod23MF}) and stacky curves (\SS\ref{sec:stacky-curves}). In Section~\ref{sec:X0N-stackyness}, we describe the stacky structures of the ``standard'' modular curves $X_{1}(N)$ and $X_{0}(N)$ over fields of all characteristics, giving careful proofs of each part of Theorem~\ref{thm:X0N-intro}. This allows us to compute rings of mod $p$ modular forms for $p = 2,3$ in Section~\ref{sec:MF23}, where we first reprove Deligne's result for $N = 1$ (\SS\ref{sec:tangent23}) and then prove Theorem~\ref{thm:levels-intro} in \SS\ref{sec:level}. We also address mod $2$ modular forms of odd weight in \SS\ref{sec:spin} and wild root stack structures in \SS\ref{sec:localstructures}. 

In Section~\ref{sec:rustom}, we prove Theorems~\ref{thm:wildVZB-intro} and~\ref{thm:wildrustom-intro}, which describe canonical rings of wild stacky curves and graded rings of mod $p$ modular forms. We then turn to analyzing ethereal modular forms in characteristics $2$ and $3$ in Section~\ref{sec:etherealMF}. Finally, in Section~\ref{sec:data} we briefly discuss modular stacky curves not of the form $X_{1}(N)$ or $X_{0}(N)$. 

The authors would like to thank Aly Deines, Kiran Kedlaya, Martin Olsson and Alice Silverberg for helpful discussions, Frank Calegari for comments on a previous draft, and K\k{e}stutis \v{C}esnavi\v{c}ius for pointing us to his article \cite{cesnavicius}. 

\subsection{Code}

Several claims in this paper are verified using the computer algebra system {\magma} \cite{Magma}.
Code verifying the computational claims made in this paper is available at the GitHub repository \cite{KobinZB:magma-scripts-modular-forms-mod-p}. 

\section{Preliminaries}
\label{sec:prelim}

In this section, we collect some background material on modular forms and stacky curves. In \SS\ref{sec:classicMF}, we briefly review the theory of complex modular forms and then in \SS\ref{sec:geoMF}, we generalize to geometric modular forms over a ring, in the style of Katz \cite{kat}. We then describe the structure of the modular curve $\X(1)$ in \SS\ref{sec:MFmodp} in order to compute the ring of modular forms in all characteristics except $2$ and $3$ -- these will be computed in Section~\ref{sec:MF23}. Finally, in \SS\ref{sec:stacky-curves} we describe the basic features of a stacky curve and prove some technical results about flat families of stacky curves. 

\subsection{Classical modular forms}
\label{sec:classicMF}

Let $\h = \{z\in\C \colon \im(z) > 0\}$ be the complex upper half-plane and define the completed upper half-plane to be $\h^{*} = \h\cup\{\infty\}\cup\Q$. 
The modular group $\Gamma(1) = SL_{2}(\Z)$ acts on $\h$ by fractional linear transformations, under which the quotient space $Y(1)^{\an} := \h/\Gamma(1)$ is isomorphic to $\C = \P_{\C}^{1}\smallsetminus\{\infty\}$, the once-punctured Riemann sphere. Its one-point compactification $X(1)^{\an} = \overline{Y(1)}^{\an}$ 
is a proper Riemann surface isomorphic to $\P_{\C}^{1}$. 

A {\it (weakly) modular function} of weight $k$ is a holomorphic function $f \colon \h\rightarrow\C$ such that for all $g = \begin{pmatrix} a & b \\ c & d\end{pmatrix}\in\Gamma(1)$, 
$$
f(z) = (cz + d)^{-k}f(gz). 
$$
Let $q = e^{2\pi iz}$. If the $q$-expansion $f(q) = \sum_{n = -\infty}^{\infty} a_{n}q^{n}$ of a modular function $f$ is holomorphic at $\infty$, i.e. $a_{n} = 0$ for all $n < 0$, then $f$ is a {\it modular form} of the same weight. {\it Cusp forms} are those modular forms whose $q$-expansions have no constant term, i.e. $a_{0} = 0$. 

For each $k\in\Z$, let $\M_{k}$ (resp. $\S_{k}$) denote the $\C$-vector space of modular forms of weight $k$ (resp. cusp forms of weight $k$). Since $-I\in SL_{2}(\Z)$, 
$\M_{k} = \{0\}$ whenever $k$ is odd. 

For a modular form $f\in\M_{k}$, with $k$ even, one may define a (holomorphic) differential form on $\h$ by
$$
\omega_{f} := f(z)\, dz^{k/2}\in\Omega_{\h/\C}^{k/2}(\h) = H^{0}(\h,\Omega_{\h/\C}^{k/2}). 
$$
Notice that for every $g\in\Gamma(1) = \SL_{2}(\Z)$, $g^{*}\omega_{f} = \omega_{f}$. Hence $\omega_{f}$ descends to a differential form on $Y(1)^{\an}$, i.e.~$\omega_{f}\in H^{0}(Y(1)^{\an},\Omega_{Y(1)^{\an}/\C}^{k})$. This computation is also compatible with the quotient structure of $X(1)^{\an}$, so we can identify $\M_{k}$ with a certain subspace of $H^{0}(X(1)^{\an},\Omega_{X(1)^{\an}/\C}^{k})$.

It is well-known (cf.~\cite[3.5]{ds}) that this subspace is described by 
\begin{equation}\label{eq:Mk}
\M_{k} = \left\{\omega\in H^{0}(X(1),\Omega_{X(1)/\C}^{k}) \left\lvert \ord_{i}(\omega)\geq -\frac{k}{4},\ord_{\rho}(\omega)\geq -\frac{k}{3},\ord_{\infty}(\omega)\geq -\frac{k}{2}\right.\right\}. 
\end{equation}
Consider the $\Q$-divisor $D = \frac{1}{2}i + \frac{2}{3}\rho + \infty$ on $X(1)^{\an}$. Then we can interpret $\M_{k}$ as the space of global sections $H^{0}(X(1)^{\an},\Omega_{X(1)^{\an}/\C}^{k}(\lfloor kD\rfloor))$. Further, \cite[Thm.~1.187]{beh} suggests we may instead view $D$ as an integral divisor on some stack $\X$ whose coarse space is $X(1)^{\an}$. Indeed, \cite[5.4.7]{vzb} provides this stacky interpretation of the above formula, which will be revisited in \SS\ref{sec:MFmodp}. 

For each (even) $k\geq 4$, let $g_{k}$ denote the weight $k$ Eisenstein series and let $\delta\in\S_{12}$ be the cusp form $\delta = (60g_{4})^{3} - 27(140g_{6})^{2}$. Then multiplication by $\delta$ gives an isomorphism $\M_{k}\rightarrow\S_{k + 12}$ for all $k\in\Z$. Using Riemann--Roch to compute $H^{0}(X(1)^{\an},\Omega_{X(1)^{\an}/\C}^{k/2}(\left\lfloor kD\right\rfloor ))$, we obtain the following well-known formula for $\dim\M_{k}$. 

\begin{prop}
\label{prop:Mkdimensionformulas}
For any $k\geq 0$, 
$$
\dim\M_{k} = \begin{cases}
  \ds\left\lfloor\frac{k}{12}\right\rfloor, & k\equiv 1\pmod{12}\\[1.5em]
  \ds\left\lfloor\frac{k}{12}\right\rfloor + 1, & k\not\equiv 1\pmod{12}. 
\end{cases}
$$
\end{prop}
%

\subsection{Geometric modular forms}
\label{sec:geoMF}

The geometric approach to modular forms over an arbitrary ring $R$ is realized by identifying an appropriate (relative) curve $X/\Spec R$ and line bundle $\L$ on $X$ and forming the graded section ring 
$$
M_{\bullet}(X,\L) = \bigoplus_{k\in\Z} H^{0}(X,\L^{\otimes k}). 
$$
For $X = X(1)^{\an}$ over $R = \C$, we already saw that $\L = \Omega^{1}\left (\frac{1}{2}i + \frac{2}{3}\rho + \infty\right )$ is a good candidate, but notice that $\L^{\otimes k}$ isn't quite the same as $\Omega^{k}(\lfloor kD\rfloor)$ used in (\ref{eq:Mk}). We will explain how to get around this issue in Example~\ref{ex:geoMFoverC} below. 

\subsubsection{The Tate curve}

For each $\tau\in\h$, set $q(\tau) = e^{2\pi i\tau}$. Then $E_{\tau} = \C/[\tau,1]$ is isomorphic to $\C^{\times}/\la q(\tau)\ra$ via the analytic map $z\mapsto q(z)$. The elliptic curve $\C^{\times}/\la q(\tau)\ra$ is called the {\it Tate curve}, written $\Tate_{\C}(q)$. Explicitly, $\Tate_{\C}(q)$ is the fiber at $\C((q))$ of an elliptic curve $\Tate_{\Z}(q)$ over $\Z((q))$ given by the affine equation 
\begin{align*}
  & y^{2} + xy = x^{3} + a(q)x + b(q)\\
  \text{where}\quad & a(q) = -5\sum_{n = 1}^{\infty} \sigma_{3}(n)q^{n} = \frac{1 - e_{4}(q)}{48}\\
  \text{and}\quad & b(q) = -\frac{1}{12}\sum_{n = 1}^{\infty} (5\sigma_{3}(n) + 7\sigma_{5}(n))q^{n} = \frac{1}{12}\left (\frac{1 - e_{4}(q)}{48} + \frac{1 - e_{6}(q)}{72}\right ). 
\end{align*}
Here, $e_{k}$ is the normalized Eisenstein series of weight $k$. 

Over an arbitrary ring $R$, the {\it Tate curve over $R$} is defined to be the base change 
$$
\Tate_{R}(q) := \Tate_{\Z}(q)\times_{\Spec\Z}\Spec R
$$
which is an elliptic curve over $R\otimes_{\Z}\Z((q))$. Write $\omega_{can} \coloneqq \pi_{*}\Omega_{\Tate_{R}(q)/R\otimes_{\Z}\Z((q))}^{1}$, where $\pi$ is the canonical projection $\Tate_{R}(q)\to\Spec(R\otimes_{\Z}\Z((q)))$. 

\subsubsection{Modular forms over arbitrary rings}

For a ring $R$, let $p \colon E\rightarrow\Spec R$ be a (relative) elliptic curve and denote by $\omega_{E/R} := p_{*}\Omega_{E/R}^{1}$ the {\it Katz canonical sheaf} of $E/R$. A {\it geometric modular function of weight $k$} over $R$ is an assignment $F$ of a section $F(E/A)\in H^{0}(A,\omega_{E/A}^{k/2})$ for every $R$-algebra $A$ and every elliptic curve $E\rightarrow\Spec A$ which satisfies: 
\begin{enumerate}[\quad (1)]
  \item $F(E/A)$ is constant on the isomorphism class of $E/A$. 
  \item If $\varphi \colon A\rightarrow B$ is a morphism of $R$-algebras and $E$ is an elliptic curve over $A$ with base change $E' = E\times_{\Spec A}\Spec B$, then $F(E'/B) = \varphi(F(E/A))$. 
\end{enumerate}
The data of a geometric modular function of weight $k$ over $R$ is equivalent to the assignment of an element $f(E/A,\omega)\in A$ to every $R$-algebra $A$, elliptic curve $E\rightarrow\Spec A$ and nonzero element $\omega\in H^{0}(E,\Omega_{E/A}^{1})$ such that: 
\begin{enumerate}[\quad (1)]
  \item $f(E/A,\omega)$ is constant on the isomorphism class of $E/A$. 
  \item For all $\alpha\in A^{\times}$, $f(E/A,\alpha\omega) = \alpha^{-k}f(E/A,\omega)$. 
  \item If $\varphi \colon A\rightarrow B$ is a morphism of $R$-algebras and $E/A$ is an elliptic curve with base change $E' = E\times_{\Spec A}\Spec B$ and compatible sections $\omega\in H^{0}(E,\Omega_{E/A}^{1})$ and $\omega'\in H^{0}(E',\Omega_{E'/B}^{1})$, then $f(E'/B,\omega') = \varphi(f(E/A,\omega))$. 
\end{enumerate}

The Tate curve allows us to define $q$-expansions geometrically: the $q$-expansion of a geometric modular function $F$ over $R$ is defined to be the section $F(\Tate_{R}(q)/R\otimes_{\Z}\Z((q)),\omega_{can})$ in $H^{0}(R\otimes_{\Z}\Z((q)),\omega_{\Tate_{R}(q)/R\otimes_{\Z}\Z((q))}) \cong R\otimes_{\Z}\Z((q))$. A {\it modular form of weight $k$} over $R$ is a geometric modular function $F$ of weight $k$ whose $q$-expansion $F(\Tate_{R}(q)/R\otimes_{\Z}\Z((q)),\omega_{can})$ lies in $R\otimes_{\Z}\Z[[q]]$. Further, $F$ is a {\it cusp form} if its $q$-expansion lies in $R\otimes_{\Z}q\Z[[q]]$. 

\begin{ex}
\label{ex:geoMFoverC}
Any geometric modular form $F$ over $\C$ determines a classical modular form $f \colon \h\rightarrow\C$ of the same weight by setting 
$$
f(\tau) = F(E_{\tau}/\C) \quad\text{(or $f(E_{\tau}/\C,dz)$ in the alternate notation)}
$$
and conversely. When $R = \C$, the affine curve $Y(1) = \Spec\C[j]$ parametrizes complex elliptic curves up to isomorphism. Suppose there were a universal elliptic curve $\pi : {\bf E}\rightarrow Y(1)$. Set $\omega_{{\bf E}/Y(1)} = \pi_{*}\Omega_{{\bf E}/Y(1)}^{1}$. Then we could identify modular functions of weight $k$ with global sections of $\omega_{{\bf E}/Y(1)}^{\otimes k/2}$ on $Y(1)$ via the Riemann existence theorem, which matches $Y(1)$ and $Y(1)^{\an}$; moreover, if $\omega_{{\bf E}/Y(1)}$ extended to a line bundle on $X(1)$, then a modular form of weight $k$ could be identified with a global section of $\omega_{{\bf E}/X(1)}^{\otimes k/2}$ on $X(1)$. The fact that no such ${\bf E}$ exists over $Y(1)$ should not deter us -- in fact, formula (\ref{eq:Mk}) suggests interpreting $\M_{k}$ as global sections over a stack. 
\end{ex}

Returning to the analytic theory for a moment, let us discuss the orbifold interpretation of classical modular forms. Let $\Y(1)^{\an} = [\h/\Gamma(1)]$ be the {\it modular orbifold curve}, which is a complex orbifold curve, i.e.~a $1$-dimensional stack over the category of complex manifolds. For each $k\in\Z$, there is a line bundle $\L_{k}$ on $\Y(1)^{\an}$ whose total space is the quotient stack $\L_{k} = [\h\times\C/\Gamma(1)]$, where $\Gamma(1)$ acts on $\h\times\C$ by 
$$
\begin{pmatrix} a & b \\ c & d\end{pmatrix}(\tau,z) = \left (\frac{a\tau + b}{c\tau + d},(c\tau + d)^{k}z\right ). 
$$
For all $k\in\Z$, the vector space of classical (weakly) modular functions on $\h$ is isomorphic to $H^{0}(\Y(1)^{\an},\L_{k})$. 

Further, there is an orbifold compactification $\X(1)^{\an} = [\h^{*}/\SL_{2}(\Z)]$ of $\Y(1)^{\an}$ which can be constructed from $\Y(1)^{\an}$ by adding an orbifold point of order $2$ at the cusp. Alternatively, one can construct $\X(1)^{\an}$ as an orbifold curve directly by gluing the affine orbifold curves $\Y(1)^{\an}$ and $[D^{2}/\mu_{2}]$ along $[\h/\la -I,T\ra] \cong [D^{2}\smallsetminus\{0\}/\mu_{2}]$, where $D^{2}$ is the complex unit disk and $\la -I,T\ra$ is the subgroup of $SL_{2}(\Z)$ generated by 
$$
-I = \begin{pmatrix} -1 & 0 \\ 0 & -1\end{pmatrix} \qquad\text{and}\qquad T = \begin{pmatrix} 1 & 1 \\ 0 & 1\end{pmatrix}. 
$$
Then $\L_{k}$ extends to $\X(1)^{\an}$, by abuse of notation also denoted $\L_{k}$, whose space of holomorphic sections $H^{0}(\X(1)^{\an},\L_{k})$ are isomorphic to $\M_{k}$, the space of modular forms of weight $k$ \cite[\SS 6.2]{vzb}. The following is well known; cf.~\cite[Exs.~3.30 and 3.55]{beh} or \cite[Ex.~5.6.12]{vzb}. 

\begin{prop}
\label{prop:P46}
There is an isomorphism of stacks $\X(1)^{\an}\cong\P(4,6)$, where $\P(4,6)$ is the weighted projective line over $\C$, considered as a $1$-dimensional stack with generic $\mu_{2}$ stabilizer. Under this isomorphism, each $\L_{k}$ is identified with $\orb(k)$. 
\end{prop}

\begin{cor}
\label{cor:stackyMFring}
The graded section ring of the line bundles $\L_{k}$ on $\X(1)$ is 
$$
\bigoplus_{k = 0}^{\infty} H^{0}(\X(1),\L_{k}) = \C[x_{4},x_{6}],
$$
where $x_{i}$ is a generator in degree $i$. 
\end{cor}

\begin{proof}
This follows from the standard fact that for any weights $a,b\in\N$, 
$$
H^{0}(\P(a,b),\orb(k)) \cong \bigoplus_{\substack{(m,n)\in\N_{0}^{2} \\ am + bn = k}} \C x^{m}y^{n}. 
$$
\end{proof}

This confirms the dimension formula in Proposition~\ref{prop:Mkdimensionformulas}, but implies much more: the graded ring of complex modular forms is given by 
\begin{equation}\label{eq:gradedMFs}
\bigoplus_{k\geq 0} \M_{k} \cong \C[e_{4},e_{6}]
\end{equation}
where $e_{k}$ is the normalized Eisenstein series of weight $k$. More importantly, a similar geometric approach allows one to compute the ring of modular forms over any base. 

Over an arbitrary field $k$, let $\Y(1)$ be the moduli stack of elliptic curves over $k$. As a moduli pseudofunctor, $\Y(1)$ sends a $k$-scheme $T$ to the groupoid $\Y(1)(T)$ of elliptic curves $E\rightarrow T$, that is, smooth, proper, pointed $T$-curves whose geometric fibres are elliptic curves. It is well-known (cf.~\cite[Ch.~13]{ols}) that $\Y(1)$ is an algebraic stack. Explicitly, $\Y(1)$ is a stack in the \'{e}tale topology admitting a smooth surjection
$$
\A_{k}^{2}\smallsetminus Z(\Delta) \longrightarrow \Y(1), 
$$
where $Z(\Delta)$ denotes the zero locus of $\Delta$ in $k[x,y]$, presenting $\Y(1)$ as an algebraic stack. When $\char k\not = 2,3$, $\Y(1) \cong [(\A_{k}^{2}\smallsetminus Z(\Delta))/\G_{m}]$ where $\G_{m}$ acts on $\A_{k}^{2}$ by $\alpha\cdot(x,y) = (\alpha^{-4}x,\alpha^{-6}y)$. Thus the coarse moduli space $Y(1)$ of $\Y(1)$ is the affine scheme $(\A_{k}^{2}\smallsetminus Z(\Delta))/\G_{m} \cong \A_{k}^{1}$. A similar argument works in $\char k = 2,3$ as well. 

Over $k = \C$, the analytification of this $Y(1)$ is the Riemann surface $Y(1)^{\an} = \h/\Gamma$ from \SS\ref{sec:classicMF}. Using a stacky version of the Riemann existence theorem \cite[Prop.~6.1.6]{vzb}, one can show that the analytic orbifold corresponding to the algebraic stack $\Y(1)$ is precisely the complex orbifold $\X(1)^{\an} = [\h/\SL_{2}(\Z)]$. 

On the algebraic side, the compactification $\X(1)$ of $\Y(1)$ constructed by Deligne and Mumford \cite{dm} is a proper stack $\X(1)$ which admits $\Y(1)$ as an open substack and extends the moduli problem of $\Y(1)$. Its precise geometric structure over different ground fields will be our key to computing the ring of modular forms in the next subsection and Section~\ref{sec:tangent23}.

\subsection{Mod $p$ modular forms}
\label{sec:MFmodp}

In this section, we describe modular forms over a field of characteristic $p > 0$. Let $\X(1)$ be the Deligne--Mumford compactification of the modular curve $\Y(1)$. Then the algebraic version of Proposition~\ref{prop:P46} holds in characteristic $p > 3$: 

\begin{thm}
\label{thm:P46modp}
Over any algebraically closed field $k$ of characteristic $\not = 2,3$, there is an isomorphism of stacks $\X(1)\cong\P(4,6)$, where $\P(4,6)$ is the weighted projective stack with weights $4$ and $6$. 
\end{thm}

This can be deduced from the fact that in characteristic $\not = 2,3$, $\X(1)$ is a $\mu_{2}$-gerbe over a tame stacky curve, hence an iterated root stack over its coarse space $\P^{1}$, together with the following computation of its automorphism groups. 

\begin{prop}[{\cite[App.~A]{sil}}]
\label{prop:ellauts}
Suppose $E$ is an elliptic curve over $\F_{q}$ where $p^{2}\mid q$ for a prime $p$, with $j$-invariant $j(E)$. Then the automorphism group $\Aut(E)$ is characterized by 
$$
\Aut(E) \cong \begin{cases}
  \Z/2\Z, & j(E)\not = 0,1728\\
  \Z/6\Z, & j(E) = 0 \text{ and } p\not = 2,3\\
  \Z/4\Z, & j(E) = 1728 \text{ and } p\not = 2,3\\
  \Z/4\Z\ltimes\Z/3\Z, & j(E) = 0 \text{ and } p = 3\\
  \Z/3\Z\ltimes Q_{8}, & j(E) = 0 \text{ and } p = 2
\end{cases}
$$
where $Q_{8}$ is the quaternion group on which $\Z/3\Z$ acts by permuting the generators $i,j,k$. Further, these descend to the following automorphism groups over $\F_{p}$: 
$$
R_{\F_{p^{2}}/\F_{p}}\Aut(E) = \Aut(E)/\{\pm 1\} \cong \begin{cases}
  \Z/2\Z, & j(E)\not = 0,1728\\
  \Z/6\Z, & j(E) = 0 \text{ and } p\equiv 1\pmod{3}\\
  \Z/3\Z, & j(E) = 0 \text{ and } p\equiv 2\pmod{3},\, p\not = 2\\
  \Z/4\Z, & j(E) = 1728 \text{ and } p\equiv 1\pmod{4}\\
  \Z/2\Z, & j(E) = 1728 \text{ and } p\equiv 3\pmod{4}\\
  S_{3}, & j(E) = 0 \text{ and } p = 3\\
  A_{4}, & j(E) = 0 \text{ and } p = 2. 
\end{cases}
$$
\end{prop}

\begin{rem}
The automorphism groups displayed in Proposition~\ref{prop:ellauts} are only the abstract groups $\Aut(E)(\overline{\F}_{p})$, not the full group schemes $\Aut(E)\rightarrow\Spec\F_{p}$. However, $\X(1)$ is a Deligne--Mumford stack by \cite{dm} and hence is characterized up to isomorphism by its coarse space, namely $X(1) \cong \P_{j}^{1}$, and its automorphism groups which are reduced, finite group schemes whose geometric points are precisely the finite groups in Proposition~\ref{prop:ellauts}. For our purposes this is enough: to compute the log canonical ring of $\X(1)$ via the stacky Riemann--Hurwitz formula \cite[Prop.~7.1]{kob}, it is enough to know the geometric automorphism groups of $\X$ and their ramification filtrations. The explicit calculations for $p = 2,3$ will be carried out in  Section~\ref{sec:tangent23}. 
\end{rem}

The connection between modular forms and $\X(1)$ is expressed by the isomorphism 
\begin{equation}\label{eq:kodairaspencer}
R(\X(1),\Delta) \cong \M(\SL_{2}(\Z))
\end{equation}
between the log canonical ring of $\X(1)$, with $\Delta$ its divisor of cusps, and the graded ring of modular forms for $\SL_{2}(\Z)$; see~\cite[Lem.~6.2.3]{vzb} or~\cite[Ex.~7.3]{kob2}. 

\begin{cor}
\label{cor:MFmodpnot23}
For any algebraically closed field $k$ of characteristic $p\not = 2,3$, the graded ring of mod $p$ modular forms is isomorphic to $k[x_{4},x_{6}]$, with $x_{i}$ in degree $i$. 
\end{cor}

\begin{proof}
By Theorem~\ref{thm:P46modp} and \cite[Prop.~5.5.6]{vzb} or \cite[Prop.~4.13]{kob}, a canonical divisor of $\X(1)$ may be taken to be 
$$
K_{\X(1)} = K_{\P^{1}} + 5P + 3Q = -2\infty + 5P + 3Q
$$
where $P$ (resp.~$Q$) is the point corresponding to elliptic curves of $j$-invariant $0$ (resp.~$1728$). Therefore a log canonical divisor is $K_{\X(1)} + \Delta =  -\infty + 5P + 3Q$ and the log canonical ring of $\X(1)$ is 
$$
R(\X(1),\Delta) = \bigoplus_{k = 0}^{\infty} H^{0}(\X(1),\orb_{\X(1)}(-\infty + 5P + 3Q)^{\otimes k/2}) \cong k[x_{4},x_{6}]. 
$$
Finally, apply formula (\ref{eq:kodairaspencer}). 
\end{proof}

This recovers formula (\ref{eq:gradedMFs}) again, as well as \cite[Prop.~6.1]{del}.

\subsection{The mod $p = 2,3$ case}
\label{sec:mod23MF}

In characteristics $2$ and $3$, the story is more complicated; cf.~\cite[Prop.~6.2]{del} and Theorems~\ref{thm:MFmod3} and~\ref{thm:MFmod2}. The stack $\X(1)$ is still Deligne--Mumford with coarse space $X(1) \cong \P_{j}^{1}$. However, the points $j = 0$ and $j = 1728$ collide in these characteristics, producing a \emph{wild} stacky point at $j = 0$ with automorphism group $\Z/4\Z\ltimes\Z/3\Z$ in characteristic $3$ and $\Z/3\Z\ltimes Q_{8}$ in characteristic $2$ (see Proposition~\ref{prop:ellauts}). By \cite[Sec.~8]{kob2}, in each case $\X(1)$ can be constructed as a fiber product of tame and wild root stacks, at least once we remove the automorphism group at the generic point. 

Let $\X(1)^{\rig}$ be the rigidification of the stack $\X(1)$, which removes this generic $\Z/2\Z$ and leaves us with a stacky curve which is birational to $\P^{1}$ \cite[Rmk.~5.6.8]{vzb}. The map $\X(1) \to \X(1)^{\rig}$ is \'{e}tale and therefore induces an isomorphism on (log) canonical rings, albeit with a change in grading. To exploit the results on stacky curves in \cite{vzb,kob,kob2}, we work with $\X(1)^{\rig}$ instead of $\X(1)$.

\begin{prop}
\label{prop:P46mod23}
Let $\X(1) = \overline{\M}_{1,1}$ be the moduli stack of elliptic curves over an algebraically closed field $k$ of characteristic $p$. Then \'{e}tale-locally, 
\begin{enumerate}[\quad (a)]
  \item If $p = 3$, $\X(1)^{\rig}$ is an Artin--Schreier root stack over a tame square root stack at $j = 0\in X(1) \cong \P^{1}$. 
  \item If $p = 2$, $\X(1)^{\rig}$ is obtained by a sequence of two Artin--Schreier root stacks over a tame cube root stack at $j = 0\in X(1) \cong \P^{1}$. 
\end{enumerate}
\end{prop}

\begin{proof}
Since $\X(1)^{\rig}$ is a stacky curve with coarse space $X(1) \cong \P^{1}$ containing a single stacky point at $j = 0 = 1728$, \'{e}tale-locally we can realize it as a quotient $[Y/G]$ where $Y\rightarrow U$ is a one-point cover of an \'{e}tale neighborhood $U$ of $j = 0$ and $G$ is the automorphism group at this point \cite[11.3.1]{ols}. When $p = 3$, the ramification filtration of $G$ is $G = G_{0} = S_{3}$, $G_{1} = \cdots = G_{m} = \Z/3\Z$, $G_{m + 1} = \{0\}$ for some ramification jump $m$ which will be computed in Section~\ref{sec:tangent23}. Thus we can apply \cite[Thm.~5.8]{kob2} to obtain the result. 

When $p = 2$, the ramification filtration starts with $G = G_{0} = A_{4}$ and $G_{1} = \Z/2\Z\times\Z/2\Z$ and the same result shows (2); see [{\it loc.~cit.}, Ex.~5.9] as well for a discussion of $(\Z/2\Z\times\Z/2\Z)$-extensions vs.~$\Z/4\Z$-extensions. 
\end{proof}

In fact, $\X(1)^{\rig}$ is a \emph{global} tame-by-wild root stack over $\P^{1}$ in characteristics $2$ and $3$. This is proven in Corollary~\ref{cor:globalASrootstructure} using modular forms. 

In any case, now that we know the local stacky structure of $\X(1)^{\rig}$ in characteristics $2$ and $3$, the next step towards a description of modular forms mod $2$ and $3$ is to compute a canonical divisor and use the Riemann--Hurwitz formula to describe the log canonical ring. This is done in Section~\ref{sec:MF23}.

\subsection{Stacky curves and degenerations}
\label{sec:stacky-curves}

Following \cite[Ch.~5]{vzb}, let $\X$ a smooth, proper, geometrically connected Deligne--Mumford stack of dimension 1 over a field $k$. If $\X$ contains a dense open subscheme then we say that $\X$ is a {\it stacky curve} and otherwise we say that $\X$ is a {\it gerby curve}.  A smooth proper morphism $\X \to S$ of stacks is a {\it relative stacky curve} if its geometric fibers are all stacky curves. We say that a stacky or gerby curve $\X/k$ (resp.~a relative stacky or gerby curve $\X/S$) is {\it tame} if its stabilizer groups all have order coprime to the characteristic of $k$ (resp.~the characteristic of each closed point of $S$); otherwise $\X$ is said to be {\it wild}.

\begin{rem}
\label{rem:rigidification}
    A gerby curve $\X$ admits a rigidification morphism $\X \to \X^{\rig}$ to a stacky curve; see \cite[App.~A]{aov}. 
\end{rem}

\begin{rem}
    A stacky curve admits a coarse space morphism; see \cite[Section 5.3 and Proposition 5.3.3]{vzb}. A tame stacky curve is a root stack over its coarse space \cite[Lemma 5.3.10]{vzb}; in other words, it is determined by its coarse space plus the location and order of its stacky points, and moreover the stabilizer groups are all cyclic and isomorphic to $\mu_n$ as group schemes.

    In contrast, wild stacky curves are not determined by their coarse space and stabilizers \cite[Remark 5.3.11]{vzb}. For example, the quotient of an Artin--Schreier curve $C$ (given locally by $y^p - y = f(x)$) by $\mathbb{F}_p$ is a wild stacky curve with coarse space $\mathbb{P}^1$ and a stabilizer group of $\mathbb{F}_p$ at infinity; as $f$ varies, the genus of $C$ grows arbitrarily, and these quotient stacks are generally non-isomorphic, cf.~\cite[Ex.~6.12 and Rmk.~6.19]{kob}. 
\end{rem}

\begin{rem}
Let $\pi \colon \X\to S$ be a relative stacky curve. In \cite[Example 11.2.2(b)]{vzb}, it is incorrectly claimed that the Euler characteristic of the fibers of a family of stacky curve is not constant, due to the possibility that stacky points collide. The example given was to consider a relative (non-stacky) curve $X \to S$ with two sections $s_1, s_2$ whose images are generically disjoint, but which intersect transversely over a fiber, and to let $\X$ be $X$  rooted at the divisors $s_1(S)$ and $s_2(S)$, say to orders $e_1$ and $e_2$, respectively.

Tim Santens pointed out to the authors of \cite{vzb} that $\X \to S$ is not smooth at the stacky points over the intersection of $s_1(S)$ and $s_2(S)$.
Indeed, such a fiber has an \'etale cover by the curve $x^{e_1} = y^{e_2}$, which is singular.
\end{rem}

In fact, tame degenerations of stacky curves disallow colliding stacky points. This will follow from invariance of the Euler characteristic of a family of stacky curves.

\begin{lem}
\label{L:degree-is-constant-in-families}
Let $\pi\colon \X \to S$ be a proper flat family of stacky curves over a connected quasi-compact base $S$ and let $\L \in \Pic \X$ be a line bundle. Then the function $s \mapsto \deg(\L|_{\X_s})$ is constant.
\end{lem}

\begin{proof}

By \cite[Theorem 10.3]{alper:Good-moduli-spaces-for-Artin-stacks} some power of $L$ descends to the coarse space of $\X$, so the statement follows from the corresponding fact for schemes 
\cite[Theorem 10.2]{fulton:intersection-theory}
\end{proof}

\begin{cor}
\label{cor:euler-char-constant}
Let $\pi \colon \X \to S$ be a relative stacky curve over a connected base $S$. Then the Euler characteristic of fibers is constant, i.e., the function $S \to \Q$ given by $s \mapsto \chi(\X_s) = \deg(\Omega^1_{\X_s})$ is constant. 
\end{cor}

\begin{proof}
This follows from Lemma \ref{L:degree-is-constant-in-families} by taking 
${\L} = \Omega^1_{\pi}$ .
\end{proof}

\begin{cor}
Let $\pi \colon \X \to S$ be a tame relative stacky curve over a connected base $S$. Suppose that $P_1$ and $P_2$ are stacky points on the generic fiber of $\pi$, with multiplicities $e_1$ and $e_2$. Then the closures $\overline{P_1}$ and $\overline{P_2}$ are disjoint.
\end{cor}

\begin{proof}
Since $\pi$ is tame, each fiber is a root stack. By the usual Euler characteristic formula for a tame stacky curve \cite[Proposition 5.5.6]{vzb} 
\[
\chi(\X_s) = 2g(\X_s) - 2 + \sum_P \frac{e_P-1}{e_P}.
\]
If two stacky points collided, since the Euler characteristic is constant, it would follow that 
\[
\frac{e_1e_2-1}{e_1e_2} = \frac{e_1-1}{e_1} + \frac{e_2-1}{e_2} =  \frac{e_1e_2-1 - (e_1-1)(e_2-1)}{e_1e_2} 
\]
which is a contradiction if $e_1,e_2 > 1$.
\end{proof}

\section{Stacky structure of modular curves}
\label{sec:X0N-stackyness}
In this section we compute the stacky structure of the ``standard'' families of modular curves, namely $X_{1}(N)$ and $X_{0}(N)$, over an algebraically closed field $K$ of arbitrary characteristic.

\subsection{$X_1(N)$ over $\C$}
Let $E/\mathbb{C}$ be an elliptic curve with $j$-invariant $1728$ and let $\Lambda = [1,i]$ be the corresponding lattice. This curve has CM by $\Q(i)$. Let $P,Q$ be the basis for $E[N]$ given by $P = 1/N, Q = i/N$. Then $i$ sends $P \mapsto Q$ and $Q \mapsto -P$, so it acts on $E[N]$ by the matrix 
$$A_i = 
\begin{pmatrix*}[r] 
0 & -1 \\ 
1 &  0
\end{pmatrix*}.$$

Similarly, let $E/\mathbb{C}$ be an elliptic curve with $j$-invariant $0$ and $\Lambda = [1,\omega]$ the corresponding lattice, where $\omega = e^{2\pi i/3}$, so that $E$ has CM by $\Q(\omega)$. Let $P,Q$ be the basis for $E[N]$ given by $P = 1/N, Q = \omega/N$. Then $\omega$ sends $P \mapsto Q$ and $Q \mapsto \omega^2/N = -P-Q$, so it acts on $E[N]$ by the matrix 
$$A_{\omega} = 
\begin{pmatrix*}[r] 
0 & -1 \\  
1 & -1\end{pmatrix*}.$$

 The matrix $A_i$ has eigenvalues $\pm i$, and in particular has no fixed points, unless $N = 2$. Even for $N = 2$, the stacky structure is ``mixed'': there are three points of exact order 2 on $E$: $P,Q$ and $P+Q$.
 The first two points are swapped by $i$, and the third point is fixed by $i$. Thus, the modular curve $X_1(2)$ is generically stacky with a $\mu_2$ stabilizer. Above $j = 1728$, it has a single stackier point with stabilizer $\mu_4$ and a second point above $j = 1728$ which has stabilizer $\mu_2$. 

Similarly, the matrix $A_{\omega}$ has eigenvalues $\omega^{\pm 1}$ and no fixed points,  unless $N = 3$, and again the stacky structure on $X_1(3)$ is ``mixed'': there are eight points of exact order 3 on E (hence at least two fixed points); the orbits
\[
P \mapsto -P-Q \mapsto Q \text{ and} 
-P \mapsto P+Q \mapsto -Q
\]
are swapped by $-I$ and thus correspond to a single non-stacky point on $X_1(3)$, while the two points $-P+Q$ and $P-Q$
are each fixed by $\omega$ and swapped by $-I$, and hence correspond to a single $\mu_3$-point on $X_1(3)$. Away from these points $X_1(3)$ is a scheme.

\subsection{$X_0(N)$ over $\C$}
\label{sec:X0N-structure}
The structure of elliptic points on $X_{0}(N)$ is more interesting, since now we want to look at fixed lines in $E[N]$ instead of points. The number of lines in $E[N]$ is $\psi(N)$, where $\psi$ is Dedekind's $\psi$-function 
$$
\psi(N) = N\prod_{p\mid N} \left (1 + \frac{1}{p}\right ). 
$$

We begin by analyzing the case when $N = \ell$ is prime, in which case $\psi(\ell) = \ell + 1$. The characteristic polynomials of $A_i$ and $A_{\omega}$ are $t^2+1$ and $t^2 + t + 1$, respectively. The first factors mod $\ell > 2$ if and only if $\ell$ is 1 mod 4, and the second factors mod $\ell > 3$ if and only if $(-3/\ell) = 1$ (if and only if $\ell$ is $1$ mod $3$). In such cases the eigenvalues are distinct. 

For $\ell > 3$ prime, $X_0(\ell)$ thus looks as follows. A generic point has generic stabilizer $\mu_2$. Above $j \neq 0,12^3$ there are $\ell+1$ points, each with stabilizer $\mu_2$. Above $j = 12^3$, if $\ell$ is 1 mod 4 then there are 2 stacky points with a $\mu_4$ stabilizer and $(\ell-1)/2$ points with a $\mu_2$ stabilizer, and if $\ell$ is $-1$ mod 4 there are $(\ell+1)/2$ points with $\mu_2$ stabilizers. Above $j = 0$, if $\ell$ is 1 mod 3 then there are 2 stacky points with a $\mu_6$ stabilizer and $(\ell-1)/3$ points with a $\mu_2$ stabilizer, and otherwise $\ell$ is $-1$ mod 3 and there are $(\ell+1)/3$ points with $\mu_2$ stabilizers. In particular, for $X_{0}(\ell)$, with $\ell > 3$ prime, there are either 0, 2, or 4 elliptic points according to the residue class of $\ell$ mod $12$. 

So for example, $X_0(11)$ is a genus $1$ curve whose points all have $\mu_2$ stabilizers, but none are elliptic points, while $X_0(13)$ is a genus $0$ curve with two elliptic points with $\mu_{4}$ stabilizers and two with $\mu_{6}$ stabilizers. Here's a table for general prime levels, for reference: 
\begin{center}
\begin{tabular}{|c|c|}
    \hline
    $\ell\mod{12}$ & orders elliptic points on $X_{0}(\ell)$\\
    \hline
    1 & $2,2,3,3$\\
    5 & $2,2$\\
    7 & $3,3$\\
    11 & no elliptic points\\
    \hline
\end{tabular}
\end{center}

For composite $N$, the elliptic points on $X_{0}(N)$ can be counted in a similar fashion, recovering the following well-known formulas (cf.~\cite[Cor.~3.7.2]{ds}). 

\begin{thm}
\label{thm:ellipticpts}
Let $N\geq 1$ and let $\ell$ denote a prime number. In characteristic $0$, the number of elliptic points on $X_{0}(N)$ with $\mu_{4}$ stabilizers is given by 
$$
\epsilon_{2}(N) = \begin{cases}
    \prod_{\text{odd } \ell\mid N} \left (1 + \legen{-1}{\ell}\right ), &\text{if } 4\nmid N\\
    0, &\text{if } 4\mid N
\end{cases}
$$
while the number of elliptic points with $\mu_{6}$ stabilizers is given by 
$$
\epsilon_{3}(N) = \begin{cases}
    \prod_{3\not = \ell\mid N} \left (1 + \legen{-3}{\ell}\right ), &\text{if } 9\nmid N\\
    0, &\text{if } 9\mid N. 
\end{cases}
$$
\end{thm}

\begin{rem}
The number $\epsilon_{2}(N)$ is equal to the number of solutions to $x^{2} + 1\equiv 0\pmod{N}$, while the number $\epsilon_{3}(N)$ is equal to the number of solutions to $x^{2} + x + 1\equiv 0\pmod{N}$. 
\end{rem}

From here on, let $\X_{0}(N)$ denote the stacky modular curve of level $\Gamma_{0}(N)$ elliptic curves and let $\X_{0}(N)^{\rig}$ denote its rigidification (Remark~\ref{rem:rigidification}). Since a stacky point of order $k$ on $\X_{0}(N)^{\rig}$ corresponds to an elliptic point on $X_{0}(N)$ with stabilizer $\mu_{2k}$, we obtain the following: 

\begin{cor}
\label{cor:stackycount}
Let $N\geq 1$. In characteristic $0$, $\X_{0}(N)^{\rig}$ is a stacky curve with coarse space $X_{0}(N)$ whose stacky locus consists of $\epsilon_{2}(N)$ stacky $\mu_{2}$-points over $j = 1728$ and $\epsilon_{3}(N)$ stacky $\mu_{3}$-points over $j = 0$, where $\epsilon_{2}(N)$ and $\epsilon_{3}(N)$ are as in Theorem~\ref{thm:ellipticpts}.
\end{cor}

\begin{rem}
Here is an alternative perspective on the counting argument above. Suppose $N = \ell$ is prime and write $\F_{\ell^{2}} = \F_{\ell}(b)$. Then the isomorphism $A = \F_{\ell^{2}}^{\times}/\F_{\ell}^{\times} \xrightarrow{\sim} \Z/(\ell + 1)\Z$ is induced by $b\mapsto 1$. A primitive $6$th root of unity in $\F_{\ell^{2}}$ is $\zeta_{6} = b^{(\ell^{2} - 1)/6}$ which maps to $0\pmod{\ell + 1}$ if $\ell\equiv 1\pmod{3}$ and $(\ell + 1)/3\pmod{\ell + 1}$ if $\ell\equiv 2\pmod{3}$. Likewise, a primitive 4th root of unity is $\zeta_{4} = b^{(\ell^{2} - 1)/4}$ which maps to $0\pmod{\ell + 1}$ if $\ell\equiv 1\pmod{4}$ and $(\ell + 1)/2\pmod{\ell + 1}$ if $\ell\equiv 3\pmod{4}$. 

For $N = \ell$ prime, the $\ell + 1$ subgroups form a torsor for the group $A = \F_{\ell^{2}}^{\times}/\F_{\ell}^{\times} \cong \Z/(\ell + 1)\Z$. If $j(\tau) = j(E) = 0$ (resp.~$1728$) then $\mu_{6}$ (resp.~$\mu_{4}$) acts on $A \cong \Z/(\ell + 1)\Z$ by $x\mapsto x + \zeta_{6}\pmod{\ell + 1}$ (resp.~$x\mapsto x + \zeta_{4}\pmod{\ell + 1}$) and one can check that there are precisely $\epsilon_{3}(\ell)$ (resp.~$\epsilon_{2}(\ell)$) fixed points under this action. 


More generally, write $N = \prod_{i = 1}^{r} q_{i}$ where $q_{1},\ldots,q_{r}$ are distinct prime powers, say $q_{i} = \ell_{i}^{a_{i}}$. Then the $\psi(N)$ cyclic subgroups of order $N$ in $E[N]$ are a torsor for the finite abelian group 
$$
A = \prod_{i = 1}^{r} \W_{a_{i}}(\F_{\ell_{i}^{2}})^{\times}/\W_{a_{i}}(\F_{\ell_{i}})^{\times}
$$
where $\W_{a}$ denotes the functor of length $a$ Witt vectors. Indeed, recall that for a prime power $q = \ell^{a}$, we have: 
\begin{itemize}
    \item $\W_{a}(\F_{q}) = \F_{q}^{a}$ as sets, so $|\W_{a}(\F_{q})| = q^{a}$. 
    \item $\W_{a}(\F_{\ell}) \cong \Z/\ell^{a}\Z$ as rings, so $|\W_{a}(\F_{\ell})^{\times}| = |(\Z/\ell^{a}\Z)^{\times}| = \ell^{a - 1}(\ell - 1)$. 
    \item $\W_{a}(\F_{\ell^{2}})^{\times} \cong \F_{\ell^{2}}^{\times}\times\W_{a - 1}(\F_{\ell^{2}})$ as abelian groups, so $|\W_{a}(\F_{\ell^{2}})^{\times}| = (\ell^{2} - 1)\ell^{2(a - 1)}$ by the first bullet point. 
\end{itemize}
Putting these together, we see that the order of $A$ is 
$$
|A| = \prod_{i = 1}^{r} \frac{|\W_{a_{i}}(\F_{\ell_{i}^{2}})^{\times}|}{|\W_{a_{i}}(\F_{\ell_{i}})^{\times}|} = \prod_{i = 1}^{r} \ell_{i}^{a_{i} - 1}(\ell_{i} + 1) = \psi(N),
$$
and one can similarly put the cosets in each factor in bijection with the $\psi(N)$ cyclic subgroups of order $N$ in $E[N] \cong (\Z/N\Z)^{2}$. One can likely prove Theorem~\ref{thm:ellipticpts} using this perspective, i.e.~by counting fixed points, but we did not work out the details. 
%
%
%
\end{rem}

\subsection{Fields $K$ with $\char K \neq 2$ or $3$}
For $\char K > 3$ not dividing $N$, the structures of $X_{1}(N)$, $X_{0}(N)$ and $\X_{0}(N)$ remain the same outside $j = 0,1728$. For an elliptic curve over $K$ with $j = 0$ or $1728$, pick an integral model $E$. Since $\char K$ does not divide $N$, $E[N]$ is \'etale, so bases remain distinct when reducing and the action of the extra automorphisms extends to the model. Thus the extra automorphisms act by the same matrices and the stacky locus remains unchanged as well. 

\subsection{Fields $K$ with $\char K = 2$ and $3$}
When $\char K = 2$ or $3$, the $j$-invariants $0$ and $1728$ collide, sometimes producing new behavior at these points on $X_{1}(N)$, $X_{0}(N)$ and $\X_{0}(N)$. In general, $X_1(N)$ is still a curve, whereas for $\X_{0}(N)$, we will see in the proof of Theorem \ref{thm:ellipticptschar23} below that non identity elements of $\Aut E$ have nontrivial eigenvalues.

\begin{thm}
\label{thm:ellipticptschar23}
In characteristic $p = 2,3$, for any $N$ not divisible by $p$, the number of elliptic points of order $2$ on $Y_{0}(N)$ is\footnote{Note that there are ramified cusps on $X_{0}(N)$, but these are accounted for by the cover $X_{0}(N)\rightarrow X(1)$ and do not contribute to the stacky locus of $\X_{0}(N)$ in Corollary~\ref{cor:stackycount23}.}
$$
\epsilon_{2}'(N) \coloneqq \begin{cases}
    \frac{\epsilon_{2}(N)}{2}, &\text{if } p = 2\\
    \epsilon_{2}(N), &\text{if } p = 3,
\end{cases}
$$
where $\epsilon_{p}(N)$ are the values from Theorem~\ref{thm:ellipticpts}. The number of elliptic points of order $3$ on $Y_{0}(N)$ is 
$$
\epsilon_{3}'(N) \coloneqq \begin{cases}
    \epsilon_{3}(N), &\text{if } p = 2\\
    \frac{\epsilon_{3}(N)}{2}, &\text{if } p = 3. 
\end{cases}
$$
\end{thm}

\begin{rem}
The main issue to resolve in the proof is that, while 
there are (separate) bases of $E[N](\overline{\mathbb{F}}_p)$ with respect to which the automorphisms $i$ and $\omega$ act by
   $$A_i = \begin{pmatrix*}[r] 0 & -1 \\ 1 & 0\end{pmatrix*} \text{ and } 
   A_{\omega} = \begin{pmatrix*}[r] 
0 & -1 \\ 
1 &  -1
\end{pmatrix*},$$
   one usually cannot choose a basis with respect to which $i$ and $\omega$ act simultaneously by these two matrices. Still, it is difficult for non-commuting matrices to have common eigenvectors, and knowledge of the character table of $\Aut E$ is sufficient for counting fixed points.
\end{rem}
 
\begin{proof}[Proof of Theorem \ref{thm:ellipticptschar23}]
The statement is multiplicative, so it suffices to prove it for $N = \ell^n$ for some prime $\ell$, and by Hensel's lemma it suffices to prove the statement for $N = \ell$ different than the characteristic. Fix an elliptic point of $Y_{0}(N)$ represented by some $(E,C)$ with $j(E) = 0$. In characteristic 3, we have that 
$$
G \coloneqq \Aut_{\overline{\mathbb{F}}_3}(E) = \langle a, b \mid ab = b^2 a, a^4 = b^3 = e\rangle \cong \Z/3\Z \rtimes \Z/4\Z.
$$
The action of $G$ on the Tate module of $E$ gives a  representation
\[
\rho \colon G \hookrightarrow \GL_2({\mathbb{Z_{\ell}}}) \hookrightarrow \GL_2(\mathbb{C}).
\]
From the known matrix representations of $i$ and $\omega$, we know that with respect to some basis, $A_i^2 = -I$ acts with trace $-2$. This is enough to identify $\rho$. Indeed, the character table of $G$ reveals that for any $1$-dimensional representation $\chi$, $\chi(i^2) = 1$; so if $\rho$ were a sum of two characters, $i^2$ would have trace 2, not $-2$. Similarly, while there are two irreducible representations of $G$ of dimension $2$, only one of them satisfies $\chi_\rho(i^2) = -2$, which determines the representation $\rho$. 

We can deduce from the character table that the stabilizers are cyclic. Indeed, one can check by hand (using the orthogonality relations) that for any non-cyclic subgroup $H$ of $G$, the restriction $\rho|_H$ remains irreducible. This was checked both by hand and with {\magma } (see the file \repolink{3.5.stabilizers-calculation.m} in \cite{KobinZB:magma-scripts-modular-forms-mod-p}).

Next, since the image of $a$ is conjugate to $A_i$, $\det a = 1$; since $ab = b^2 a$, we also have that $\det b = 1$. Since $G$ acts via $2$-by-$2$ matrices whose trace and determinant are known, we also know their characteristic polynomials. There are, respectively, 1,1,2,6,2 elements of orders 1,2,3,4,6 with characteristic polynomials 
\[
(t-1)^2,(t+1)^2,t^2+t+1,t^2+1,t^2+t-1
\]
and eigenvalues $1,-1, \omega^{\pm 1}, \pm i$, and $-\omega^{\pm 1}$; they are diagonalizable if and only if $\ell$ is $1 \mod 4$ (for $\pm i$) and  $1 \mod 3$ (for $\pm \omega$). 

The three pairs $a^{\pm 1}, a^{\pm 1}b$ and $a^{\pm 1}b^2$ of order 4 elements have distinct eigenspaces, but $\langle b \rangle$ permutes these pairs of eigenspaces (i.e., if $\overline{v}$ is an eigenvector of $a$, then $b\overline{v}$ is an eigenvector of $bab^{-1} = ab$). These three pairs of eigenspaces are thus in the same $G$-orbit, and therefore determine the same point on $Y_{0}(N)$. Moreover, since there is no element $c \in G$ such that $cac^{-1} = a^{-1}$ (and in particular, no automorphism swapping the eigenspaces of $a$), we conclude that there are two elliptic points of order 2. 

Similarly, let $c$ and $c^{-1}$ be the distinct elements of order 6. There are thus at most 2 elliptic points of order 3 (corresponding to the two distinct and common eigenspaces of $c$ and $c^{-1}$); but since $aca^{-1} = c^{-1}$, these two eigenspaces are swapped, and thus in the same $G$ orbit (therefore the same moduli point). We conclude that there is one elliptic point of order 3, concluding the proof for $p = 3$.

Now let $p = 2$; then 
$$
G = \Aut_{\overline{\mathbb{F}}_2}(E) \cong Q_8 \rtimes \Z/3\Z
$$
where $Q_8$ is the quaternion group, and again
the action of $G$ on the Tate module of $E$ gives a 2-dimensional representation $\rho$. The same argument as for $p = 3$ implies that the representation is irreducible, and the character table of $G$ reveals that there are 3 irreducible representations of $G$ of dimension $2$, the traces of known matrix representations of $i$ and $\omega$ determine $\rho$, the relations imply that $G$ lands in $\SL_2$, and this determines the characteristic polynomials. There are, respectively, 1,1,8,6,8 elements of orders 1,2,3,4,6 with characteristic polynomials 
\[
(t-1)^2,(t+1)^2,t^2+t+1,t^2+1,t^2+t-1
\]
and eigenvalues $1, -1, \omega^{\pm 1}, \pm i$, and $-\omega^{\pm 1}$, again diagonalizable if and only if $\ell$ is $1 \mod 4$ (for $\pm i$) and  $1 \mod 3$ (for $\pm \omega$).

Finally, the elements of order 4 are all conjugate; in particular, each such element $a$ is conjugate to its inverse via an element that swaps the eigenspaces of $a$. We conclude that there is one elliptic point of order 2. On the other hand, there are two conjugacy classes of elements of order six, and each such element is conjugate to its inverse (again via an element that swaps its eigenspaces); we conclude that there are two elliptic points of order 3. See the {\magma} file \repolink{3.5.stabilizers-calculation.m} in \cite{KobinZB:magma-scripts-modular-forms-mod-p} for additional verifications of some of the claims in this proof.
 \end{proof}

\begin{cor}
\label{cor:stackycount23}
In characteristic $p = 2,3$, $\X_{0}(N)^{\rig}$ is a stacky curve with coarse space $X_{0}(N)$ whose stacky locus consists of $\epsilon_{2}'(N)$ stacky $\Z/2\Z$-points and $\epsilon_{3}'(N)$ stacky $\Z/3\Z$-points over $j = 0$, where $\epsilon_{2}'(N)$ and $\epsilon_{3}'(N)$ are as in Theorem~\ref{thm:ellipticptschar23}. 
\end{cor}

   \begin{rem} Since $\dim \rho = 2$, it is possible to deduce explicit matrix representatives for $i$ and $\omega$ from the character $\chi_\rho$. For example, let 
  \[
  A = \begin{pmatrix*}[r] 0 & -1 \\ 1 & 0\end{pmatrix*} \text{ and } B = \begin{pmatrix*}[r] a & b \\ c & d\end{pmatrix*} \in \GL_2(\mathbb{Z}/N\mathbb{Z})
  \]
  be matrices such that $AB = B^2A$ and $B^3 = I$. Then $b = c$, $d = -(a+1)$ and $a^2 + a + 1 + b^2 = 0$. Indeed, while the character table does not give matrix representatives $A$ and $B$ of $i$ and $\omega$, after choosing a basis such that $i$ acts by $A$, the following three facts determine $B$:
\begin{enumerate}
    \item since $\tr B = -1$, $a+d = -1$;
    \item since $\tr AB = \tr \begin{pmatrix*}[r] -c & -d \\ a & b\end{pmatrix*} = 0$, $b = c$; and
    \item since $ABA^{-1} = B^2$, $\det B = 1$ and thus  $(a+1)a + b^2 + 1 = 0$.
\end{enumerate}
A similar argument works for $\Aut_{\overline{\mathbb{F}}_2}(E)$. 

On the other hand, in contrast to the case when $\char K > 3$, it is not possible to find \emph{integer} matrices $A$ and $B$ which work for all $\ell$; this is clear from appearance of the conic $a^2 + a + 1 + b^2 = 0$. The conic also suggests a quaternion algebra hiding in the background, and a comment by Will Sawin on this MathOverflow post \cite{chen:mo-automorphisms}
further suggests that $i$ and $\omega$ act by elements of an algebraic group $G/\mathbb{Z}$ (related to a quaternion algebra) such that $G \otimes \mathbb{Z}_{\ell}$ is isomorphic to $\SL_2(\mathbb{Z}_{\ell})$ for $\ell > 3$, but such that $G$ is not isomorphic to $\SL_2$. Sawin further suggests that $G$ should be the group of norm one elements in a maximal order of a quaternion algebra ramified at 2 or 3.
   \end{rem}


\section{Modular forms in characteristic $2$ and $3$}
\label{sec:MF23}
In this section, we compute the graded ring of modular forms mod $p$ for $p = 2,3$ by computing the log canonical ring of $\X(1)$ in those characteristics. From Proposition~\ref{prop:P46mod23}, we know $\X(1)$ is a $\Z/2\Z$-gerbe over the stacky curve $\X(1)^{\rig}$ with a single wild stacky point at $j = 0$. To find a canonical divisor, we will use the wild stacky Riemann--Hurwitz formula \cite[Prop.~7.1]{kob}, which says that for a stacky curve $\X$ with coarse space isomorphic to $\P^{1}$ and a single stacky point $P$, 
$$
K_{\X} = -2H + \sum_{i = 0}^{\infty} (|G_{i}| - 1)P
$$
where $H\not = P$ and $G_{i}$ is the ramification filtration at $P$. For $\X = \X(1)^{\rig}$, rather than computing the groups $G_{i}$ directly, we instead construct an \'{e}tale cover of $\X(1)^{\rig}$ and pull $K_{\X}$ back to the cover to deduce the ramification jumps at $P$. Note that Corollary~\ref{cor:euler-char-constant} also allows us to deduce $\deg(K_{\X(1)^{\rig}})$ directly, but the \'{e}tale cover method will be useful in more general situations. 

We then use our methods to compute the ring of mod $p$ modular forms with level $N$, for $N$ not divisible by $p$, proving Theorem~\ref{thm:levels-intro}. We will also describe odd weight modular forms in characteristic $2$ in \SS\ref{sec:spin} and the global root stack structure of $\X(1)$ in \SS\ref{sec:localstructures}.

\subsection{Tangent bundles in characteristic $2$ and $3$}
\label{sec:tangent23}
Let $\ell$ be a prime and let $K$ be a field of characteristic different from $\ell$.
Let $X(\ell)$ be the modular curve defined in \cite[4.1]{poonenSS:Twists-of-X7-and-primitive-solutions-to}; over a non-algebraically closed field this differs slightly from the ``usual" $X(\ell)$, and paramaterizes pairs $(E,\iota)$, where $E$ is an elliptic curve and $\iota$ is a symplectic isomorphism
\[
\iota \colon \mu_{\ell} \times \Z/\ell \Z \xrightarrow{\;\sim\;} E[\ell],
\]
where ``symplectic" means that $\iota$ respects the Weil pairing, i.e., the composition
\[
\mu_\ell \cong \bigwedge\!\!^2\left(\mu_{\ell} \times \Z/\ell \Z\right)  \xrightarrow{\;\sim\;} \bigwedge\!\!^2 E[\ell] \cong \mu_{\ell}
\]
is the identity (where the first isomorphism is canonical, and the last is induced by the Weil pairing). If $K$ contains $\mu_{\ell}$ then $X(\ell)$ is isomorphic to the ``usual" $X(\ell)$ (or, depending on one's definition of $X(\ell)$, isomorphic to a connected component of $X(\ell)$). Over such a field $K$, pre-composition gives an action of $\PSL_2(\mathbf{F}_\ell)$ (and over a general field $K$, an action of a twisted version of $\PSL_2(\mathbf{F}_\ell)$; see \cite[4.2]{poonenSS:Twists-of-X7-and-primitive-solutions-to}). The quotient stack $\X = [X(\ell)/\PSL_2(\mathbf{F}_\ell)]$ is a stacky $\P^{1}$, with generically trivial stabilizer, a stabilizer of order $\ell$ at $\infty$, and in characteristic 2 or 3 a stabilizer of order $6$ at $j = 0$. 

As in Section~\ref{sec:MFmodp}, let $\X(1)^{\rig}$ be the rigidification of the stack $\X(1)$ by $\Z/2\Z$; then $\X(1)^{\rig}$ is a stacky $\P^1$ with a single stacky point of order $6$ at $j = 0$. 

Our goal is to compute the canonical divisor $K_{\X(1)^{\rig}}$ of $\X(1)^{\rig}$. We have maps
\[
X(\ell) \xrightarrow{\pi} \mathcal{X} \xrightarrow{\phi} \X(1)^{\rig}.
\]
The map $\phi$ is given by rooting (tamely) at infinity to order $\ell$, while the map $\pi$ is \'{e}tale. In particular, 
\begin{equation}\label{eq:etale-cover-canonical}
\pi^*K_{\mathcal{X}} = K_{X(\ell)}
\end{equation}
and
\[
\phi^*K_{\X(1)^{\rig}} + (\ell - 1)\infty = K_{\mathcal{X}}.
\]
 Since $\X(1)^{\rig}$ has only one stacky point, this is enough information to compute $\phi^*K_{\X(1)^{\rig}}$, as we now explain. Absorbing the rest of the canonical divisor (i.e., the ``$-2\infty$'') via linear equivalence, 
\[
K_{\X(1)^{\rig}} = a[0:1]
\]
for some $a$. Then 
\[
K_{\mathcal{X}} = a[0:1] + (\ell - 1)\infty.
\]
Taking degrees of (\ref{eq:etale-cover-canonical}), we have 
\[
\deg K_{X(\ell)} = \deg \pi^*K_{\mathcal{X}} = \#\PSL_2(\mathbf{F}_\ell)\cdot\deg K_{\mathcal{X}}.
\]
Putting this all together, we get
\begin{equation}
\label{eq:Xellcover}
2g(X(\ell)) - 2= \#\PSL_2(\mathbf{F}_\ell)\left (\frac{a}{6} + \frac{\ell-1}{\ell}\right ).
\end{equation}
Solving for $a$ gives
\[
a = 6\left (\frac{2g(X(\ell)) - 2}{\#\PSL_2(\mathbf{F}_\ell)} - \frac{\ell-1}{\ell}\right ).
\]
Taking $\ell = 7$ and $g(X(7)) = 3$ gives $a = -5$. As a check, taking $\ell = 11$ and $g(X(11)) = 26$ also gives $a = -5$. This gives an explicit demonstration of the fact that the degree of $K_{\X(1)^{\rig}}$ must be constant, as in Lemma~\ref{L:degree-is-constant-in-families}. 

\begin{rem}
\label{rem:jumpmod3}
By \cite[Prop.~7.1]{kob}, the canonical divisor on $\X(1)^{\rig}$ is also given by $K_{\X(1)^{\rig}} = K_{X(1)} + \frac{b}{6}[0 : 1]$, where 
$$
b = \sum_{i = 0}^{\infty} (|G_{i}| - 1) = (6 - 1) + (3 - 1)m = 5 + 2m
$$
if the ramification filtration at the stacky point $[0 : 1]$ has jump $m$. Our calculations above show that $b = 7$ and so $m = 1$. Another version of this calculation is given in \cite[Sec.~11]{dardayasuda}. 
\end{rem}

Finally, this is the canonical divisor, and not the log canonical divisor; adding in the contribution from the cusp gives a log canonical divisor of the form 
\[
K_{\X(1)^{\rig}} + \Delta = [0:1]
\]
with $\deg([0 : 1]) = \frac{1}{6}$. 

\begin{thm}[{\cite[Prop.~6.2(II)]{del}}]
\label{thm:MFmod3}
The graded ring of modular forms mod $3$ is isomorphic to $k[x_{2},x_{12}]$, where $x_{i}$ is a generator in degree $i$. 
\end{thm}

\begin{proof}
We have $\deg(K_{\X(1)^{\rig}} + \Delta) = \frac{1}{6}$, so by \cite[Thm.~3.4]{ODorn}, 
\[
R_{[0:1]/6} \cong k[x_1,x_6]
\]
where $\deg x_i = i$. Rigidification only changes the grading by a factor $2$, so the log canonical ring of $\X(1)$ is 
$$
R(\X(1),\Delta) \cong k[x_{2},x_{12}]. 
$$
Finally, apply \cite[Lem.~6.2.3]{vzb}.  
\end{proof}

\begin{rem}
\label{rem:mod3gens}
The Hasse invariant $a = a_{3}$ is a mod $3$ modular form of weight $2$, so we may choose $x_{2} = A$ in the description above. In particular, the usual Eisenstein series generators $e_{4}$ and $e_{6}$ of $R(\X(1),\Delta)$ are generated from the Hasse invariant: 
$$
e_{4} = a^{2} \quad\text{and}\quad e_{6} = a^{3}. 
$$
(A priori these only hold up to constants, but reducing the $q$-expansions of $e_{4}$ and $e_{6}$ mod $3$ shows these identities directly.) Similarly, we may take $x_{12}$ to be the modular discriminant. 
\end{rem}

The computation above works nearly identically in characteristic $2$. This time, $\X(1)^{\rig}$ is a stacky $\P^{1}$ with a stacky point of order $12$ at $j = 0$, so 
$$
K_{\X(1)^{\rig}} = a[0 : 1]
$$
for some $a$, with $\deg([0 : 1]) = \frac{1}{12}$. Then formula (\ref{eq:Xellcover}) becomes 
$$
2g(X(\ell)) - 2 = \#\PSL_{2}(\F_{\ell})\left (\frac{a}{12} - \frac{\ell - 1}{\ell}\right )
$$
and any choice of $\ell$ will produce $a = -10$ (again guaranteed by Lemma~\ref{L:degree-is-constant-in-families}), so $K_{\X(1)^{\rig}} + \Delta = 2[0:1]$ and the same proof as for Theorem~\ref{thm:MFmod3} shows: 

\begin{thm}[{\cite[Prop.~6.2(I)]{del}}]
\label{thm:MFmod2}
The graded ring of even weight modular forms mod $2$ is isomorphic to $k[x_{2},x_{12}]$, where $x_{i}$ is a generator in degree $i$. 
\end{thm}

\begin{rem}
\label{rem:jumpmod2}
As in Remark~\ref{rem:jumpmod3}, we can use the formulas $K_{\X(1)^{\rig}} = K_{X(1)} + c[0 : 1]$ and 
$$
c = \sum_{i = 0}^{\infty} (|G_{i}| - 1) = (12 - 1) + (4 - 1)m = 11 + 3m
$$
to deduce that in characteristic $2$, the ramification jump\footnote{A priori, there are two ramification jumps for a wild ramification group of order $4$, but the small value of $c$ forces them to be equal. Note that this is only possible because $G_{1} \cong \Z/2\Z\times\Z/2\Z$ and not $\Z/4\Z$.} at the degree $\frac{1}{12}$ stacky point $[0 : 1]$ on $\X(1)^{\rig}$ is $m = 1$. 
\end{rem}

In the statement of Theorem~\ref{thm:MFmod2}, we emphasized that $k[x_{2},x_{12}]$ is the ring of \emph{even weight} mod $2$ modular forms because, in contrast to the classical situation, \emph{there are odd weight modular forms in characteristic $2$}. In particular, the Hasse invariant $a = a_{2}$ is a mod $2$ modular form of weight $1$, so it gives us a new generator $x_{1}$ in the ring of modular forms. This is emphasized by the formula $K_{\X(1)^{\rig}} + \Delta = \frac{2}{12}[0 : 1]$, which shows that this stacky curve has a ``half log canonical divisor'' $E = [0 : 1]$ of degree $\frac{1}{12}$. We give a full description of odd weight forms in Section~\ref{sec:spin}. 

The intermediate stack $\X$ used in the arguments above is an additional example of a wild stacky curve which was previously studied in \cite{bcg}. 

\begin{ex}
\label{ex:PSL2quotientstack}
Let $\ell > 5$ be prime and consider the quotient stack $\X = [X(\ell)/\PSL_{2}(\F_{\ell})]$. In characteristic $3$, this is a stacky curve with underlying $\P^{1}$ and two stacky points $P$ and $Q$, with automorphism groups $\mu_{\ell}$ and $S_{3}$; cf.~\cite[Lem.~3.1]{bcg} or \cite[Rmk.~5.3.11]{vzb}. Thus $Q$ is a wild stacky point. 

By \cite[Lem.~3.1]{bcg}, or using the computation above, the ramification jump at $Q$ is $1$. Then by \cite[Prop.~7.1]{kob}, a canonical divisor on $\X$ is given by 
$$
K_{\X} = -2H + (\ell - 1)P + 7Q. 
$$
The ramification jump can be confirmed by pulling the canonical divisor back to $X(\ell)$ along the quotient map. By \cite[Cor.~7.3]{kob}, the genus of $\X$ is 
$$
g(\X) = \frac{13\ell - 6}{12\ell}
$$
which grows with $\ell$ even though the coarse space has genus $0$ in all characteristics. 

To compute the canonical ring of $\X$, we can replace $K_{\X}$ with the linearly equivalent divisor $(\ell - 1)P - 5Q$ and apply \cite[Thm.~4.2]{ODorn} to see that $R(\X)$ is generated in degrees $6\leq d\leq\ell$ with relations in degrees up to $2\ell$. Notice that when $\ell = 2$ (resp.~$\ell = 5$), $\X$ is still a wild stacky curve with genus $\frac{5}{6}$ (resp.~$\frac{59}{60}$) but since $K_{\X}$ is non-effective, the canonical ring is trivial. 
\end{ex}

\begin{ex}
The same curve $\X$ in characteristic $2$ also exhibits wild behavior: it is a stacky $\P^{1}$ whose stacky locus consists of a tame point $P$ with automorphism group $\mu_{\ell}$ and a wild point $Q$ with automorphism group $A_{4}$ and ramification jump $1$ \cite[Lem.~3.1]{bcg}. Then \cite[Prop.~7.1]{kob} gives us a canonical divisor 
$$
K_{\X} = -2H + (\ell - 1)P + 14Q
$$
and the same genus formula as above. Therefore the canonical ring has the same description as above for $\ell\geq 7$. For $\ell = 3$ (resp.~$\ell = 5$), $\X$ is a wild stacky curve with genus $\frac{11}{12}$ (resp.~$\frac{59}{60}$) and trivial canonical ring. 
\end{ex}

\subsection{Level structure}
\label{sec:level}

Consider the stacky modular curve $\X_{0}(N)$ for $N\geq 5$ in characteristic $p = 2$ or $3$. We expect the stack structure to change in one of two ways (see Section~\ref{sec:X0N-stackyness}): 
\begin{itemize}
    \item the stacky points can collide, as with $\X(1)^{\rig}$ (corresponding to points over $j = 0$ and $j = 1728$ colliding); or 
    \item the automorphism groups at some points (stacky or non-stacky) can grow. 
\end{itemize}
Before giving a general result (Theorem~\ref{thm:levels}), we first analyze some examples with prime level $\ell$. 

\begin{ex}
\label{ex:X05mod2}
Let's start with $\X_{0}(5)$ in characteristic $2$. By Corollary~\ref{cor:stackycount}, in tame characteristic, there are two stacky points above $j = 0$ with automorphism group $\mu_2$, and above $j = 1728$ there are two stacky points with automorphism group $\mu_4$; in characteristic $2$, all four of these stacky points collide in a single point. After rigidification, we are left with a single stacky point with automorphism group $\Z/2\Z$ (Corollary~\ref{cor:stackycount23}). In particular, the point must have a ramification jump, so the log canonical divisor will have larger coefficients than in tame characteristics and there may be ethereal modular forms. 

By \cite[II.2]{maz}, the map on coarse spaces 
$$
X_{1}(5) \rightarrow X_{0}(5)
$$
is ramified exactly at the unique point $P_{0}$ above $j = 0 = 1728$ with inertia group $\Z/2\Z$ and ramification jump $m = 1$, which is confirmed by the following calculation. The map factors as 
$$
X_{1}(5) \xrightarrow{\pi} \X_{0}(5)^{\rig} \xrightarrow{\phi} X_{0}(5)
$$
where $\pi$ is \'{e}tale of degree $\frac{1}{2}[\Gamma_{0}(5) : \Gamma_{1}(5)] = 2$ and $\phi$ is a wild square root stack at $P_{0}$ with jump $m$. Let $P$ denote the stacky point above $P_{0}$, which has degree $\frac{1}{2}$. Since $\pi$ is \'{e}tale, $K_{X_{1}(5)} = \pi^{*}K_{\X_{0}(5)^{\rig}}$. Using \cite[Prop.~7.1]{kob}, we get 
$$
K_{\X_{0}(5)^{\rig}} = \phi^{*}K_{X_{0}(5)} + (m + 1)(2 - 1)\phi^{*}P_{0} = \phi^{*}K_{X_{0}(5)} + (m + 1)P. 
$$
Then 
$$
-2 = \deg K_{X_{1}(5)} = \deg\pi\cdot\deg K_{\X_{0}(5)^{\rig}} = 2\deg(-2\infty + (m + 1)P) = -4 + (m + 1)
$$
which confirms $m = 1$. So a log canonical divisor on $\X_{0}(5)^{\rig}$ is 
$$
K_{\X_{0}(5)^{\rig}} + \Delta = 2P
$$
(using $\deg(\Delta) = 2$) and the ring of mod $2$ modular forms of level $\Gamma_{0}(5)$ has two generators in degree $2$ (after reindexing/unrigidifying). 

Compare this to the ring of classical modular forms of level $\Gamma_{0}(5)$, which has only one generator in degree $2$. This can be seen by computing a log canonical divisor $K_{\X_{0}(5)^{\rig}} + \Delta = P_{1} + P_{2}$, where $P_{1}$ and $P_{2}$ both have degree $\frac{1}{2}$. In particular, the mod $2$ reduction map fails to be surjective on level $\Gamma_{0}(5)$ modular forms, so one of the generators is ethereal. 

In characteristic $3$, interesting things happen geometrically: points above $j = 0$ and $1728$ collide (see Section~\ref{sec:X0N-stackyness}) so that $\X_{0}(5)^{\rig}$ has two (tame) stacky $\mu_{2}$-points above $j = 0$, say $P_{1}$ and $P_{2}$. However, $P_{1} + P_{2}$ is still a log canonical divisor, so the ring of modular forms is unchanged from the tame case. 
\end{ex}

\begin{ex}
\label{ex:X013mod2}
Consider $\X_{0}(13)$, which also has genus $0$ coarse space. In tame characteristic, the stackiness consists of two $\mu_{6}$-points above $j = 0$ and two $\mu_{4}$-points above $j = 1728$, which collide in characteristic $2$. Rigidifying and reducing mod $2$ produces a stacky curve with a wild $\Z/2\Z$-point $P$ and two $\mu_{3}$-points $Q_{1}$ and $Q_{2}$, all over $j = 0$. (That is, two $\mu_{2}$-points collide into a wild $\Z/2\Z$-point, while the tame stacky points remain distinct.) Let their images in $X_{0}(13)$ be denoted $P_{0},Q_{10}$ and $Q_{20}$, respectively. As above, the ramified cover $X_{1}(13)\rightarrow X_{0}(13)$ factors as 
$$
X_{1}(13) \xrightarrow{\pi} \X_{0}(13)^{\rig} \xrightarrow{\phi} X_{0}(13)
$$
with $\pi$ \'{e}tale of degree $6$ and $\phi$ is locally a wild square root at $P_{0}$, say with jump $m$, and a cube root at each of $Q_{10}$ and $Q_{20}$. By \cite[Prop.~7.1]{kob}, 
$$
K_{\X_{0}(13)^{\rig}} = \phi^{*}K_{X_{0}(13)} + (m + 1)P + 2Q_{1} + 2Q_{2},
$$
while $\deg K_{X_{1}(13)} = 2g(X_{1}(13)) - 2 = 2$, so we have 
$$
2 = 6\deg(-2\infty + (m + 1)P + 2Q_{1} + 2Q_{2}) = 3m - 1. 
$$
Thus $m = 1$, a log canonical divisor for $\X_{0}(13)^{\rig}$ is 
$$
K_{\X_{0}(13)^{\rig}} + \Delta = 2P + 2Q_{1} + 2Q_{2},
$$
and by Riemann--Roch \cite[Cor.~1.189]{beh}, $h^{0}(\X_{0}(13)^{\rig},K_{\X_{0}(13)^{\rig}} + \Delta) = 2$. Meanwhile, in tame characteristics, $K_{\X_{0}(13)^{\rig}} + \Delta = P_{1} + P_{2} + 2Q_{1} + 2Q_{2}$ is a log canonical divisor, with $\deg P_{1} = \deg P_{2} = \frac{1}{2}$ and $\deg Q_{1} = \deg Q_{2} = \frac{1}{3}$, so $h^{0}(\X_{0}(13)^{\rig},K_{\X_{0}(13)^{\rig}} + \Delta) = 1$. Once again, there is an ethereal modular form in weight $2$ (after reindexing). 
\end{ex}

\begin{ex}
\label{ex:X013mod3}
In characteristic $3$, $\X_{0}(13)$ behaves differently: the two $\mu_{6}$-points collide, producing a rigidification with genus $0$ coarse space and stackiness concentrated at a $\Z/3\Z$-point, say $P$, and two $\mu_{2}$-points, $Q_{1}$ and $Q_{2}$. Let $m$ be the ramification jump at $P$. Similar calculations show that $m = 1$, so a log canonical divisor is 
$$
K_{\X_{0}(13)^{\rig}} + \Delta = 4P + Q_{1} + Q_{2}. 
$$
Then $h^{0}(\X_{0}(13)^{\rig},K_{\X_{0}(13)^{\rig}} + \Delta) = 2$, so there is an ethereal form in weight $2$ in the ring of mod $3$ modular forms of level $\Gamma_{0}(13)$. 
\end{ex}


In general, we have the following characterization of the stacky structure of $\X_{0}(N)^{\rig}$ in characteristic $p\nmid N$. 

\begin{thm}
\label{thm:levels}
Let $N > 1$. Then 
\begin{enumerate}[(1)]
    \item If $2\nmid N$, the following are equivalent:
    \begin{enumerate}[(a)]
        \item In characteristic $2$, $\X_{0}(N)^{\rig}$ has a stacky $\Z/2\Z$-point over $j = 0$. 
        \item In characteristic $2$, $\X_{0}(N)^{\rig}$ has fewer stacky $\Z/2\Z$-points than it has stacky $\mu_{2}$-points in characteristic $0$. 
        \item $N$ is a product of primes congruent to $1$ mod $4$. 
        \item The ring of mod $2$ modular forms of level $\Gamma_{0}(N)$ has an ethereal form in weight $2$. 
    \end{enumerate}
    \item If $3\nmid N$, the following are equivalent: 
    \begin{enumerate}[(a)]
        \item In characteristic $3$, $\X_{0}(N)^{\rig}$ has a stacky $\Z/3\Z$-point over $j = 0$. 
        \item In characteristic $3$, $\X_{0}(N)^{\rig}$ has fewer stacky $\Z/3\Z$-points than it has stacky $\mu_{3}$-points in characteristic $0$. 
        \item $N$ is a product of primes congruent to $1$ mod $3$. 
        \item The ring of mod $3$ modular forms of level $\Gamma_{0}(N)$ has an ethereal form in weight $2$. 
    \end{enumerate}
    \item If $3\nmid N$, the following are equivalent: 
    \begin{enumerate}[(a)]
        \item In characteristic $0$, $\X_{0}(N)^{\rig}$ has $r > 0$ stacky $\mu_{2}$-points over $j = 1728$. 
        \item In characteristic $3$, $\X_{0}(N)^{\rig}$ has $r > 0$ stacky $\mu_{2}$-points over $j = 0$. 
        \item $N$ is a product of primes congruent to $1$ mod $4$ or $2$ times such a product. 
    \end{enumerate}
    \item In characteristic $2$, no stacky $\mu_{3}$-points collide or appear. 
\end{enumerate}
\end{thm}

\begin{rem}
The condition in (1c) is equivalent to $N$ being a primitive sum of squares, i.e.~$N = x^{2} + y^{2}$ for relatively prime $x,y$. Similarly, the condition in (2c) is equivalent to $N$ being primitively represented by the quadratic form $N = x^{2} + xy + y^{2}$, and the condition in (3c) is equivalent to $N$ being a primitive Pythagorean hypotenuse, i.e.~$2N = x^{2} + y^{2}$ for relatively prime $x,y$. 
\end{rem}

\begin{proof}
The implications (a) $\iff$ (b) $\iff$ (c) in (1) -- (3), as well as statement (4), all follow from Corollaries~\ref{cor:stackycount} and~\ref{cor:stackycount23}. 

(1b) $\iff$ (1d) In characteristic $0$, $\X_{0}(N)^{\rig}$ has a canonical divisor of the form 
$$
K \coloneqq K_{\X_{0}(N)^{\rig}} = \phi^{*}K_{X_{0}(N)} + P_{1} + \ldots + P_{r} + 2Q_{1} + \ldots + 2Q_{s}
$$
where $P_{1},\ldots,P_{r}$ are the stacky points of degree $\frac{1}{2}$, necessarily lying over $j = 1728$, and $Q_{1},\ldots,Q_{s}$ are the stacky points of degree $\frac{1}{3}$, lying over $j = 0$. If there are no such points, we interpret this as $r = 0$ or $s = 0$, appropriately. Then the dimension of the Riemann--Roch space $H^{0}(\X_{0}(N)^{\rig},K + \Delta)$ is 
$$
2g(X_{0}(N)) - 2 + r\left\lfloor\frac{1}{2}\right\rfloor + s\left\lfloor\frac{2}{3}\right\rfloor + \deg\Delta + 1 = 2g(X_{0}(N)) - 2 + \deg\Delta + 1,
$$
since $K + \Delta$ is effective. Reducing mod $2$ or $3$, the points collide over $j = 0 = 1728$. 

Mod $2$, all stacky $\Z/2\Z$-points over $j = 0$ come from stacky points over $j = 1728$ in characteristic $0$. If there are fewer of these, i.e.~$r$ decreases, then at least two of them come together, say $P_{1}$ and $P_{2}$ collide into a wild point $P$, and the canonical divisor now has a term $(m + 1)P$, where $m$ is the ramification jump at $P$. Since $m\geq 1$ for wildly ramified points, we have a term $\left\lfloor\frac{m + 1}{2}\right\rfloor \geq 1$ in the formula for $h^{0}(\X_{0}(N)^{\rig},K + \Delta)$ and hence this dimension increases, which is equivalent to the existence of an ethereal modular form in weight $2$. Conversely, such a form increases the dimension of the Riemann--Roch space, which can only happen if the divisor of floors of degree $\frac{1}{2}$ stacky points changes. Indeed, $\phi^{*}K_{X_{0}(N)}$, $\Delta$ and $2Q_{1} + \ldots + 2Q_{s}$ are unchanged in characteristic $2$, so $\lfloor P_{1}\rfloor + \ldots + \lfloor P_{r}\rfloor$ must change. The only way for \emph{this} to happen is for some pair $P_{i},P_{j}$ to collide, causing $\lfloor P_{i}\rfloor + \lfloor P_{j}\rfloor$ to be replaced by $\lfloor (m + 1)P\rfloor$, with $m\geq 1$. By \cite[Prop.~7.1]{kob}, $m$ is the ramification jump at $P$, so $P$ must be wild. 

(2b) $\iff$ (2d) Mod $3$, the argument is similar: all stacky $\Z/3\Z$-points over $j = 0$ come from stacky points over $j = 0$ in characteristic $0$ and if $s$ decreases, we get a term $2(m + 1)Q$ in the canonical divisor, with $m\geq 1$. Thus $\left\lfloor\frac{2(m + 1)}{3}\right\rfloor \geq 1$ in the dimension formula for $h^{0}(\X_{0}(N)^{\rig},K + \Delta)$, so this dimension increases, which is equivalent to the existence of an ethereal modular form in weight $2$. The converse is similar. 
\end{proof}

In an unpublished article \cite{cesnavicius}, \v{C}esnavi\v{c}ius studies mod $p$ reductions of the modular curves $\X_{0}(N)$ and proves part of Theorem~\ref{thm:levels}(1) in [{\it loc.~cit.}, Lem.~3.17(b)]. It is likely that his method also works in characteristic $3$, which gives an alternative argument for some parts of Theorem~\ref{thm:levels}. However, the proof does not include the wildly ramified cases and it appears \v{C}esnavi\v{c}ius was not aware of the non-liftable forms in $\M_{2}(\Gamma_{0}(N);\F_{p})$ for $p = 2$ and $3$ (see [{\it loc.~cit.}, Rmk.~3.19]). 

Note that Theorem~\ref{thm:levels} also extends the wild cases in the table in \cite[II.2]{maz} to composite levels. 

\begin{thm}
\label{thm:levelforms}
If $N = p_{1}^{a_{1}}\cdots p_{r}^{a_{r}}$ for distinct primes $p_{i}\equiv 1\pmod{4}$, then $\X_{0}(N)^{\rig}$ has $2^{r}$ stacky $\mu_{2}$-points over $j = 1728$ in tame characteristics which collide into $2^{r - 1}$ wild stacky $\Z/2\Z$-points over $j = 0$ mod $2$ and this produces $2^{r - 1}$ linearly independent ethereal forms in weight $2$. Likewise, if $N = q_{1}^{b_{1}}\cdots q_{s}^{b_{s}}$ for distinct primes $q_{i}\equiv 1\pmod{3}$, then there are $2^{s}$ stacky $\mu_{3}$-points over $j = 0$ which collide into $2^{s - 1}$ wild stacky $\Z/3\Z$-points mod $3$, producing $2^{s - 1}$ linearly independent ethereal forms in weight $2$. 
\end{thm}

\begin{proof}
The point counts follow from Corollaries~\ref{cor:stackycount} and~\ref{cor:stackycount23} and the number of linearly independent ethereal forms is an immediate consequence of the wild Riemann--Hurwitz formula \cite[Prop.~7.1]{kob} and Riemann--Roch. 
\end{proof}

Other ``standard'' level structures also exhibit interesting behavior in characteristics $2$ and $3$, though they need not produce ethereal modular forms. 

\begin{ex}
In tame characteristics, the moduli stack $\X_{1}(3)$ of level $\Gamma_{1}(3)$ elliptic curves is a stacky curve with a single stacky $\mu_{3}$-point above $j = 0$ and a non-stacky point above $j = 0$ (see Section~\ref{sec:X0N-stackyness}), hence is a (tame) root stack over its genus $0$ coarse space $X_{1}(3)$. It therefore has log canonical divisor 
$$
K_{\X_{1}(3)} + \Delta = 2P,
$$
where $\deg P = \frac{1}{3}$, and the ring of modular forms of level $\Gamma_{1}(3)$ is generated in degrees $1,2$ and $3$. In characteristic $2$, the two points above $j = 0$ collide to produce a single stacky $\mu_{3}$-point. In any case, the log canonical divisor is the same so the ring of modular forms remains unchanged. 
\end{ex}

\begin{ex}
Similarly, consider the stack $\X_{1}(2)$. In tame characteristics, this is a gerby curve with generic automorphism group $\mu_{2}$ and a single ``stackier'' point above $j = 1728$ with automorphism group $\mu_{4}$. Its rigidification $\X_{1}(2)^{\rig}$ is a stacky curve with a single stacky $\mu_{2}$-point over $j = 1728$, so a log canonical divisor is 
$$
K_{\X_{1}(2)^{\rig}} + \Delta = Q
$$
where $\deg Q = \frac{1}{2}$. This means the ring of modular forms of level $\Gamma_{1}(2)$ is generated in degrees $2$ and $4$. In characteristic $3$, the points above $j = 0$ and $j = 1728$ collide to produce a single stacky $\mu_{4}$-point, but the rigidification has the same canonical divisor, so the ring of modular forms remains unchanged. 
\end{ex}

For an example of a modular curve with ``nonstandard'' level structure, see Section~\ref{sec:data}. 


\subsection{Spin divisors}
\label{sec:spin}

Classically, there are no modular forms of odd weight and level $\Gamma$ whenever $-I\in\Gamma$, since the transformation law $f(z) = (-1)^{-k}f(-Iz)$ implies $f\equiv 0$ for odd $k$. In particular, there are no odd weight modular forms for the full modular group $\SL_{2}(\Z)$. However, when $\Gamma$ does not contain $-I$, there may be odd weight forms, even in weight $1$. These correspond to sections of powers of the line bundle associated to a {\it spin log canonical divisor}, also called a {\it half log canonical divisor} or {\it theta characteristic}; cf.~\cite[Ch.~10]{vzb}. 

Such a spin divisor arises when a log canonical divisor $D = K_{\X(\Gamma)} + \Delta$ has even degree in $\Pic(\X(\Gamma))$, where $\X(\Gamma)$ is the moduli stack of elliptic curves with $\Gamma$ level structure. Equivalently, a theta characteristic is a line bundle $\L$ on $\X(\Gamma)$ such that $\L^{\otimes 2}\cong \orb(D)$. 

However, not every square root of $\orb(D)$ has sections which are odd weight modular forms. In the classical case, where $\X(\Gamma)$ can be replaced with a complex orbifold $\X(\Gamma)^{\an}$, there is a unique choice of $\L$ (the Hodge bundle) such that the spin log canonical ring $R(\X(\Gamma),\Delta,\L)$ coincides with the graded ring of modular forms of all weights \cite[Lem.~10.2.2]{vzb}. 

In our setting, where $\X = \X(\Gamma)$ is a moduli stack of elliptic curves in characteristic $p$, the algebraic Hodge bundle $\L$ is again a theta characteristic and the identification $R(\X(\Gamma),\Delta,L) \cong \M(\Gamma)$ still holds by the theorem of Kodaira--Spencer; cf.~\cite[A.1.3.17]{kat} or~\cite[Rem.~10.2.3]{vzb}. 

\begin{thm}
The full graded ring of mod $2$ modular forms is $k[x_1,x_{12}]$, where $x_i$ is a generator in degree $i$. 
\end{thm}

\begin{proof}
By \cite{mum,fulton-olsson}, $\Pic(\X(1)) = \Z/12\Z$ so our log canonical divisor $D = \frac{1}{6}[0 : 1]$ from Section~\ref{sec:tangent23} has a spin divisor $E = \frac{1}{12}[0 : 1]$; equivalently, the log canonical bundle $\orb(D)$ has a square root $L$, which is unique in this case  up to order $2$ elements in Pic. The proof in \cite[\SS 6]{mum} also shows that $\L$ is a generator of $\Pic(\X(1))$, so we may identify $\L \cong L = \orb(E)$. The theorem then follows from the Kodaira--Spencer isomorphism. 
\end{proof}

\begin{rem}
This demonstrates a counterexample to the paragraph at the end of \cite[Sec.~10.1]{vzb} in the wild case: in characteristic $2$, the order of the automorphism group of the unique stacky point of $\X(1)^{\rig}$ is $6$, but the presence of wild ramification (specifically, an odd ramification jump) at this point still allows for a spin log canonical divisor. 
\end{rem}

\begin{rem}
As in Remark~\ref{rem:mod3gens}, we may choose $x_{1} = a_{2}$, the mod $2$ Hasse invariant, and $x_{12} = \delta$, the modular discriminant. Then one can check that the Eisenstein series $e_{4}$ and $e_{6}$ satisfy $e_{4} = x_{1}^{4}$ and $e_{6} = x_{1}^{6}$ mod $2$. 
\end{rem}

\begin{rem}
In contrast, there are no odd weight mod $3$ modular forms since no spin divisor of $D = \frac{1}{6}[0 : 1]$ has global sections. 
\end{rem}

For $N > 1$, $\X_{0}(N)^{\rig}$ also has a spin log canonical divisor and weight $1$ modular forms in characteristic $2$. For simplicity, we will restrict our focus to even weight forms for the rest of the present article.

\subsection{Root stack structures}
\label{sec:localstructures}

For $\X(1)^{\rig}$ in characteristics $p = 2,3$, we saw in Proposition~\ref{prop:P46mod23} that the stacky structure is concentrated at $j = 0$. We also deduced in Remarks~\ref{rem:jumpmod3} and~\ref{rem:jumpmod2} that the ramification jump in each case is $m = 1$. So far, we have only used the \'{e}tale-local structure of this stacky point, namely as a local tame-by-wild root stack, to compute rings of modular forms. It turns out that $\X(1)^{\rig}$ is a \emph{global} tame-by-wild root stack, which can be proven using our description of the canonical ring. 

We illustrate this structure in the $p = 3$ case and then point out the differences when $p = 2$. Let $j \colon X(1)\rightarrow\P^{1} = \Proj k[x_{0},x_{1}]$ be the $j$-map, corresponding to the data $(L,s,f) = (\L^{\otimes 18},e_{4}^{9},\delta^{3} - e_{4}^{6}\delta)$, where $\L$ is the Hodge bundle, $e_{4}$ is the Eisenstein series in degree $2$ and $\delta$ is the modular discriminant in degree $6$, after the grading shift (see Remark~\ref{rem:mod3gens}). To exhibit the global structure of $\X(1)^{\rig}$, we first take a tame root along $(L,e_{4}^{9})$, then a wild root along $(L^{\otimes 1/2},s^{1/2},f) = (\L^{9},e_{6}^{3},f)$. 

For the square root, construct the pullback squares 
\begin{center}
\begin{tikzpicture}[xscale=3.5,yscale=2]
    \node at (0,0) (a0) {$X(1)$};
    \node at (1,0) (b0) {$\P^{1}$};
    \node at (2,0) (c0) {$[\A^{1}/\G_{m}]$};
    \node at (0,1) (a1) {$X(1)'$};
    \node at (1,1) (b1) {$\P(2,1)$};
    \node at (2,1) (c1) {$[\A^{1}/\G_{m}]$};
    \draw[->] (a0) -- (b0) node[above,pos=.5] {$(L,s,f)$};
    \draw[->] (b0) -- (c0) node[above,pos=.5] {$(\orb(1),x_{0})$};
    \draw[->] (a1) -- (a0) node[left,pos=.5] {$\pi_{1}$};
    \draw[->] (b1) -- (b0);
    \draw[->] (c1) -- (c0) node[right,pos=.5] {$(-)^{2}$};
    \draw[->] (a1) -- (b1) node[above,pos=.5] {$(L_{1},s_{1},f_{1})$};
    \draw[->] (b1) -- (c1) node[above,pos=.5] {$\left (\orb\left (\frac{1}{2}\right ),x_{0}^{1/2}\right )$};
\end{tikzpicture}
\end{center}
Here, $L = \L^{\otimes 18}$ can be identified with $\orb_{X(1)}(1)$, so that 
$$
L_{1} = \pi_{1}^{*}L^{\otimes 1/2} = \orb_{X(1)'}\left (\tfrac{1}{2}\right ) = \L^{\otimes 9}, \quad s_{1} = \pi_{1}^{*}s^{1/2} = e_{4}^{9/2} = e_{6}^{3} \quad\text{and}\quad f_{1} = \pi_{1}^{*}f.
$$
Next, the map $X(1)'\rightarrow\P(2,1)$ specified by $(L_{1},s_{1},f_{1})$ can be pulled back along the universal Artin--Schreier root stack with jump $m = 1$ from \cite[Defn.~6.8]{kob} to give
\begin{center}
\begin{tikzpicture}[xscale=3.5,yscale=2]
    \node at (0,0) (a0) {$X(1)'$};
    \node at (1,0) (b0) {$\P(2,1)$};
    \node at (0,1) (a1) {$X(1)''$};
    \node at (1,1) (b1) {$\P(2,1)$};
    \draw[->] (a0) -- (b0) node[above,pos=.5] {$(L_{1}',s_{1}',f_{1}')$};
    \draw[->] (a1) -- (a0) node[left,pos=.5] {$\pi_{2}$};
    \draw[->] (b1) -- (b0) node[right,pos=.5] {$\wp_{1}$};
    \draw[->] (a1) -- (b1) node[above,pos=.5] {$(L_{2},s_{2},f_{2})$};
\end{tikzpicture}
\end{center}
We claim that $X(1)'' \cong \X(1)^{\rig}$. By construction, $\pi_{2}^{*}L_{1} \cong L_{2}^{\otimes 3}$, $\pi_{2}^{*}s_{1} = s_{2}^{3}$ and $\pi_{2}^{*}f_{1} = f_{2}^{3} - f_{2}s_{2}^{2}$. But $s_{1} = \pi_{1}^{*}s^{1/2}$ and $f_{1} = \pi_{1}^{*}f^{1/2}$ so that 
$$
L_{2} \cong \pi^{*}L^{1/6} = \pi^{*}\L^{3}, \quad s_{2}^{3} = \pi^{*}s^{1/2} = e_{6}^{3}\quad\text{and}\quad \pi^{*}f^{1/2} = f_{2}^{3} - f_{2}\pi^{*}s^{1/3}
$$
Identifying $L_{2}$ with the log tricanonical bundle $\orb_{\X(1)^{\rig}}(3(K + \Delta)) = \orb_{\X(1)^{\rig}}\left (3[0 : 1]\right )$, $s_{2}$ with the cube of the Hasse invariant and $f_{2}$ with the form $\delta^{3} - e_{4}^{6}\delta$ gives the isomorphism. 

For $p = 2$, the tame root stack is along $(L,s) = (\L^{\otimes 12},e_{4}^{6})$ and the wild root stack is along $(L_{1},s_{1},f_{1}) = (\L^{\otimes 4},e_{4}^{2},\delta^{2} + e_{4}^{3}\delta)$. At the end, the isomorphism $X(1)'' \cong \X(1)^{\rig}$ is established by identifying $s_{2}$ with the square of the Hasse invariant and $f_{2}$ with the given weight $12$ modular form. 

This proves: 
\begin{cor}
\label{cor:globalASrootstructure}
In characteristics $2$ and $3$, the moduli stack $\X(1)^{\rig}$ is a global tame-by-wild root stack over $X(1)\cong\P^{1}$. Explicitly: 
\begin{enumerate}[\quad (1)]
    \item When $p = 3$, $\X(1)^{\rig}$ is isomorphic to an Artin--Schreier root stack over the tame root stack $\P(2,1)$. 
    \item When $p = 2$, $\X(1)^{\rig}$ is obtained by a sequence of two Artin--Schreier root stacks over the tame root stack $\P(3,1)$. 
\end{enumerate}
\end{cor}

A similar argument can be used to characterize the wild root stack structures of the $\X_{0}(N)$, using an explicit presentation of the ring of mod $p$ modular forms which we will obtain in Section~\ref{sec:rustom}. 

For example, in characteristic $2$, $\X_{0}(5)$ is a stacky $\P^{1}$ with a single wild $\Z/2\Z$-point at $j = 0$. We will see in Example~\ref{ex:MFmod2level5} that its ring of modular forms is generated by two forms $x$ and $y$ in weight $2$ such that $y^{2} + xy$ is the mod $2$ reduction of a classical form of weight $4$. This Artin--Schreier relation expresses the global root stack structure of $\X_{0}(5)$ over its coarse space.

\section{Rustom's Conjecture in characteristics $2$ and $3$}
\label{sec:rustom}
For many practical purposes, it is extremely useful to have a description of the graded ring of modular forms of level $\Gamma_{0}(N)$ in terms of generators and relations in low degrees. In tame characteristics, i.e.~characteristic $0$ or $p > 0$ for $p\nmid 6N$, a uniform description was conjectured by Rustom \cite[Conj.~2]{rus} and proved by Voight and the second author \cite{vzb}: 

\begin{thm}[{\cite[Cor.~1.5.1]{vzb}}]
\label{thm:rustom-tame}
For $N\geq 1$, the graded ring of modular forms 
$$
M_{\bullet}\left (N;\Z\left [\tfrac{1}{6N}\right ]\right ) = \bigoplus_{k = 0}^{\infty} M_{k}\left (N;\Z\left [\tfrac{1}{6N}\right ]\right )
$$
has a presentation with generators in weights $\leq 6$ and relations in weights $\leq 12$. 
\end{thm}

\begin{rem}
\label{rem:rustomfibers}
To prove this theorem, it is sufficient by \cite[Lem.~11.2.5]{vzb} to consider each fiber $\X_{0}(N)^{\rig}_{\Q}$ and $\X_{0}(N)^{\rig}_{\F_{p}}$ for $p\nmid 6N$ and apply the main theorem of [{\it loc.~cit.}], which bounds the degrees of generators and relations in the log canonical ring of a tame stacky curve. 
\end{rem}

As stated, the conjecture is false over $\Z\left [\frac{1}{N}\right ]$ precisely because the fibers of $\X_{0}(N)^{\rig}$ may be wild and therefore the main result in [{\it loc.~cit}] does not apply. 

\begin{ex}
\label{ex:rustomfalse}
For level $N = 1$, Rustom's conjecture is false over $\Z$: in characteristics $2$ and $3$, there is a generator in weight $12$ by Theorems~\ref{thm:MFmod3} and~\ref{thm:MFmod2}. 
Nevertheless, Theorems~\ref{thm:MFmod3} and~\ref{thm:MFmod2} and Remark~\ref{rem:rustomfibers} suggest the following modified version of Rustom's conjecture. 
\end{ex}

\begin{thm}[Rustom's Conjecture - Wild Case]
\label{thm:wildrustom}
For $N\geq 1$, the graded ring of modular forms $M_{\bullet}\left (N;\Z\left [\frac{1}{N}\right ]\right )$ has a presentation with generators and relations in weights $\leq 12$. Moreover, for $N > 1$, the generators appear in weights $\leq 6$ with relations in weights $\leq 12$, as in the tame case. 
\end{thm}


\begin{ex}
\label{ex:MFmod2level5}
(\repolink{5.5.ring-of-mod-forms.m} in \cite{KobinZB:magma-scripts-modular-forms-mod-p})
We saw in Example~\ref{ex:X05mod2} that in tame characteristics, $\X_{0}(5)^{\rig}$ is a stacky genus $0$ curve with signature is $(0;2,2;2)$. This means the canonical ring has a presentation with generators (red) in weights $\leq 4$ with a unique minimal relation (blue) in weight $8$. 

Mod $2$, $\X_{0}(5)^{\rig}$ reduces to a wild stacky curve with a unique stacky point $P$ of degree $\frac{1}{2}$, so its log canonical divisor becomes $2P$. The naive signature of $\X_{0}(5)^{\rig}$ is $(0;2;2)$, so \cite[Thm.~1.4.1]{vzb} would predict that the ring of mod $2$ level $5$ modular forms is generated in degree at most $4$. Accounting for the ramification jump though, this log stacky curve behaves more like one of naive signature $(0;1;2)$, with canonical ring $\overline{\F}_{2}[x_{2},y_{2}]$. In any case, we see that Rustom's conjecture still holds for $N = 5$ in characteristic $2$ despite the presence of the ethereal generator $y_{2}$. 
\end{ex}

\begin{ex}
\label{ex:MFmod3level7}
(\repolink{5.6.ring-of-mod-forms.m} in \cite{KobinZB:magma-scripts-modular-forms-mod-p})
Similarly, in characteristic $p\not = 3$, $\X_{0}(7)^{\rig}$ is a stacky curve with signature $(0;3,3;2)$ and its log canonical ring is generated forms in weights $2,4,4,6,6$ with relations in degrees $8,8,8,10,10,12$. 

Mod $3$, $\X_{0}(7)^{\rig}$ reduces to a stacky curve with log canonical divisor of the form $4Q$, where $\deg(Q) = \frac{1}{3}$. This structure falls outside the framework of \cite{vzb}; in particular, it doesn't make sense to assign a naive signature to $\X_{0}(7)^{\rig}_{\overline{\F}_{3}}$ since the coefficient of the stacky point $Q$ in $\lfloor K_{\X_{0}(7)^{\rig}}\rfloor$ is not of the form $\frac{e - 1}{e}$. Nevertheless, Magma shows that $M_{\bullet}(7;\overline{\F}_{3})$ has a presentation with generators in weights $2,2,6$ and a single relation in weight $8$ (see also \cite[Thm.~4]{ODorn}). So Rustom's conjecture still holds for $N = 7$ in characteristic $3$. 
\end{ex}

\begin{ex}
\label{ex:MFmod2level13}
(\repolink{5.7.ring-of-mod-forms.m} in \cite{KobinZB:magma-scripts-modular-forms-mod-p})
Consider $\X_{0}(13)^{\rig}$. In tame characteristics, its signature is $(0;2,2,3,3;2)$ and its log canonical ring has a presentation with generators in weights $2,4,4,4,4,6,6$ and relations in weights up to $12$. 

Mod $2$, $\X_{0}(13)^{\rig}$ picks up wild stacky points, resulting in a log canonical divisor of the form $2P + 2Q_{1} + 2Q_{2}$, with $\deg(P) = \frac{1}{2}$ and $\deg(Q_{i}) = \frac{1}{3}$. The ring of modular forms still satisfies the bounds in Rustom's conjecture though, as there is a presentation with generators in weights $2,2,4,4,6,6$ and a simpler set of relations, though still in weights up to $12$. 
\end{ex}

\begin{ex}
\label{ex:MFmod3level13}
(\repolink{5.7.ring-of-mod-forms.m} in \cite{KobinZB:magma-scripts-modular-forms-mod-p})
The same curve $\X_{0}(13)^{\rig}$ reduces mod $3$ to a stacky $\P^{1}$ with log canonical divisor $P_{1} + P_{2} + 4Q$, with $\deg(P_{i}) = \frac{1}{2}$ and $\deg(Q) = \frac{1}{3}$. The ring of modular forms in this case has a presentation with generators in weights $2,2,4,4,6$ and relations in weights up to $10$, further simplifying the presentation in characteristic $0$. 
\end{ex}

\subsection{Log canonical rings for wild stacky curves}

\begin{defn}
For a log stacky curve $(\X,\Delta)$ over a field $k$, with coarse moduli map $\pi \colon \X\rightarrow X$ and stacky points $P_{1},\ldots,P_{r}$, we define the {\bf refined signature} of $\X$ to be the tuple $(g;c_{1},\ldots,c_{r};\delta)$ where $g = g(X)$ is the genus of the coarse space $X$ of $\X$, $c_{i}$ is the rational coefficient of $\pi(P_{i})$ in the pushforward $\pi_{*}K_{\X}$ of the canonical divisor, and $\delta = \deg(\Delta)$. 
\end{defn}

\begin{rem}
The refined signature differs from the signature defined in \cite{vzb} as follows. For a tame log stacky curve $(\X,\Delta)$ of signature $(g;e_{1},\ldots,e_{r};\delta)$, the refined signature of $(\X,\Delta)$ is $\left (g;\frac{e_{1} - 1}{e_{1}},\ldots,\frac{e_{r} - 1}{e_{r}};\delta\right )$. More generally, \cite[Prop.~7.1]{kob} shows that 
$$
c_{i} = \sum_{j = 0}^{\infty} \frac{|G_{P_{i},j}| - 1}{|G_{P_{i}}|}
$$
where $G_{P_{i}}$ is the automorphism group at $P_{i}$ and $G_{P_{i},j}$ are its ramification subgroups in the lower numbering. 
\end{rem}


Recall the statement of the main theorem in \cite{vzb}: 
\begin{thm}[{\cite[Thms.~8.4.1 and 9.3.1]{vzb}}]
For a tame, separably rooted log stacky curve $(\X,\Delta)$ with naive signature $(g;e_{1},\ldots,e_{r};\delta)$, set $e = \max\{e_{1},\ldots,e_{r}\}$. Then the log canonical ring is generated in degrees $\leq 3e$ with relations in degrees $\leq 6e$. If $g + \delta\geq 2$, the bounds are $\leq \max(3,e)$ and $\leq 2\max(3,e)$, respectively. 
\end{thm}

Our main theorem in this section extends this to the wild setting. 

\begin{thm}
\label{thm:wildVZB}
Let $(\X,\Delta)$ be a (possibly wild) separably rooted log stacky curve with refined signature $(g;c_{1},\ldots,c_{r};\delta)$, coarse space $\pi\colon\X\to X$ and stacky points $P_{1},\ldots,P_{r}$. Let $e_{i}$ be the denominator of $c_{i}$ (in lowest terms) and set $e = \max\{e_{1},\ldots,e_{r}\}$ and $c = \sum_{i = 1}^{r} \lfloor c_{i}\rfloor$. Then the log canonical ring $R(\X,\Delta)$ is generated in degrees $\leq 3e$ with relations in degrees $\leq 6e$. Moreover, if $g + c + \delta\geq 2$, the bounds are $\leq \max(3,e)$ and $\leq 2\max(3,e)$, respectively. 
\end{thm}

\begin{proof}
The theorem will follow once we verify the following statements: 
\begin{enumerate}[\quad (i)]
    \item If $g + c + \delta\geq 2$, then $R(\X,\Delta)$ is generated in degrees $\leq \max(3,e)$ with relations in degrees $\leq 2\max(3,e)$. 
    \item $c > 0$ if and only if some $P_{i}$ is wild. 
    \item The remaining cases have $g = 0$, $\delta = 0$ and $c = 1$, i.e.~no log divisor and exactly one wild stacky point, and have $R(\X,\Delta)$ generated in degrees $\leq 3e$ with relations in degrees $\leq 6e$. 
\end{enumerate}

(i) The key is \cite[Thm.~8.3.1]{vzb}, which needs refinement in the stacky case. Let $\X\to\X'$ be a tame-by-wild root stack (see \cite{kob2}, especially \SS 5.1) over a stacky curve $\X'$ with refined signature $(g;c_{1},\ldots,c_{r - 1};\delta)$, rooted precisely at one nonstacky point $P$ on $\X'$. Factoring through an intermediate stacky curve if necessary, we may assume $P_{r}$ is totally wildly ramified. 
By induction, we may further assume $\X'$ is a tame stacky curve. Write $c_{r} = n_{r} + \frac{d_{r}}{e_{r}}$ with $n_{r}\in\Z_{> 0}$ and $0\leq d_{r} < e_{r}$. Then for all $k\geq 2$, $\lfloor kc_{r}P_{r}\rfloor = kn_{r}\pi(P_{r}) + \left\lfloor\frac{kd_{r}}{e_{r}}P_{r}\right\rfloor$ and we can simply move the $kn_{r}\pi(P_{r})$ term into $k\Delta$. Thus we may replace $(\X,\Delta)$ with a log stacky curve with refined signature $\left (g;c_{1},\ldots,c_{r - 1},\frac{d_{r}}{e_{r}};\delta + n_{r}\right )$. Likewise, replace $(\X',\Delta)$ with a tame log stacky curve with refined signature $(g;c_{1},\ldots,c_{r - 1};\delta + n_{r})$. 

Since $\X'$ is tame, \cite[8.3.2]{vzb} shows that for all $k\geq 2$, 
$$
\deg\lfloor k(K_{\X} + \Delta)\rfloor = \deg\lfloor k(K_{\X'} + \Delta)\rfloor + \left\lfloor\frac{kd_{r}}{e_{r}}\right\rfloor \geq 2g - 1
$$
and as in [{\it loc.~cit.}], the divisors $\lfloor k(K_{\X} + \Delta)\rfloor$ are nonspecial. By \cite[Thm.~8.3.1(a)]{vzb}, there are elements $y_{k}\in H^{0}(\X,k(K_{\X} + \Delta)) = H^{0}(X,\lfloor k(K_{\X} + \Delta)\rfloor)$ for $2\leq k\leq e_{r}$ which generate $R(\X,\Delta)$ as an $R(\X',\Delta)$-algebra. The rest of the proof of \cite[Thm.~8.3.1]{vzb} goes through as written, with the following small modification: 
\begin{center}
replace $\left\lfloor (d + 1)\left (1 - \frac{1}{e}\right )\right\rfloor - \left\lfloor d\left (1 - \frac{1}{e}\right )\right\rfloor \leq 1$ with $\left\lfloor (d + 1)\frac{d_{r}}{e_{r}}\right\rfloor - \left\lfloor d\frac{d_{r}}{e_{r}}\right\rfloor \leq 1$. 
\end{center}
Finally, combining this with the second statement in \cite[Thm.~8.4.1]{vzb} for $\X'$ gives the result. 

(ii) It follows from \cite[Prop.~7.1]{kob} that $c = 0$ if and only if $\X$ is tame, in which case the main theorem of \cite{vzb} applies. Notice that this and (i) imply that all cases where $g\geq 1$ are handled. 

(iii) We may assume the same setup as in (i), i.e.~with $(\X,\Delta)\to (\X',\Delta)$ where $\X'$ is tame of refined signature $(g;c_{1},\ldots,c_{r - 1};\delta + n_{r}) = \left (0;\frac{1}{e_{1}},\ldots,\frac{1}{e_{r - 1}};1\right )$ and where $P_{r}$ is wild with coefficient $c_{r} = \frac{d_{r}}{e_{r}}$. Such a curve $\X'$ is handled by \cite[Ch.~9]{vzb}, and the inductive step in (i) works here as well, so we are done. 
\end{proof}

Applying this to the stacky modular curves $\X_{0}(N)^{\rig}$ for $N > 1$, together with the computations in \SS\ref{sec:X0N-structure}, proves the wild version of Rustom's conjecture (Theorem~\ref{thm:wildrustom}).

\subsection{Base Cases}
\label{sec:rustom-examples}

In \cite{vzb}, one of the main challenges in proving the main theorem was to check a number of ``base cases'' whose coarse spaces had genus $g = 0$. Wild ramification makes the analogous base cases easier, as one can see in the proof of Theorem~\ref{thm:wildVZB}(iii). We illustrate this further here. 

Suppose $\X$ is a wild stacky curve with refined signature $(0;c_{1},\ldots,c_{r};0)$ such that $c = \sum_{i = 1}^{r} \lfloor c_{i}\rfloor = 1$. The condition $c = 1$ imposes restrictions on the ramification jumps at the stacky point. For instance, a stacky $\P^{1}$ with a single stacky $\Z/p\Z$-point has refined signature $\left (0;\frac{(p - 1)(m + 1)}{p};0\right )$ by \cite[Ex.~7.2]{kob}, where $m$ is the ramification jump. The only valid cases where $c = 1$ are those for which 
$$
p\leq (p - 1)(m + 1) < 2p
$$
which forces\footnote{In the $p = 2$ case, we can rule out $m = 2$ since we know $\gcd(m,p) = 1$.} $m = 1$. But the canonical ring is trivial in these cases since 
$$
\deg\lfloor kK_{\X}\rfloor = -2k + \left\lfloor\frac{k(p - 1)(m + 1)}{p}\right\rfloor = -2k + \left\lfloor\frac{2k(p - 1)}{p}\right\rfloor < 0
$$
for all $k\geq 1$. 

For a stacky point of order $p^{2}$ with $c = 1$, the two ramification jumps $m_{1}$ and $m_{2}$ must satisfy 
$$
p^{2}\leq (p^{2} - 1)(m_{1} + 1) + (p - 1)(m_{2} - m_{1}) < 2p^{2}. 
$$
This forces $m_{1} = m_{2} = 1$, which is only possible if the automorphism group at this point is $\Z/p\Z\times\Z/p\Z$. Then for all $k\geq 1$, 
\begin{align*}
    \deg\lfloor kK_{\X}\rfloor &= -2k + \left\lfloor\frac{k((p^{2} - 1)(m_{1} + 1) + (p - 1)(m_{2} - m_{1}))}{p^{2}}\right\rfloor\\
        &= -2k + \left\lfloor\frac{2k(p^{2} - 1)}{p^{2}}\right\rfloor < 0. 
\end{align*}
So once again, the canonical ring is trivial. The proof of the general case is similar, and we deduce: 

\begin{prop}
Let $\X$ be a stacky curve with coarse space $\P^{1}$ and a single wild stacky point with $c = 1$. Then the canonical ring $R(\X)$ is trivial. 
\end{prop}

For the same curve $\X$, if $\delta\geq 1$, then $k(K_{\X} + \Delta)$ has sections and 
$$
\deg\lfloor k(K_{\X} + \Delta)\rfloor = k(\delta - 2) + \left\lfloor\frac{k(p - 1)(m + 1)}{p}\right\rfloor \geq 0
$$
for all $1\leq k\leq p$, but these are no longer base cases in the sense that they are handled by Theorem~\ref{thm:wildVZB}(i). For example, we saw in Example~\ref{ex:MFmod2level5} that when $p = 2$, $\delta = 2$ and $m = 1$, the log canonical ring is generated in degree $1$ with no relations. 

\begin{ex}(\repolink{5.13.can-ring.m} in \cite{KobinZB:magma-scripts-modular-forms-mod-p})
In this example, we illustrate how wild ramification simplifies the presentation of the log canonical ring of one of the base cases from \cite[Ch.~9]{vzb}. Let $\X$ be a $\Z/5\Z$ Artin--Schreier root stack over a stacky curve with naive signature $(0;2,3,7;0)$, say with jump $m = 1$ at the wild point. Then the refined signature of $\X$ is $\left (0;\frac{1}{2},\frac{2}{3},\frac{6}{7},\frac{8}{5};0\right )$ and its canonical ring is generated in degrees $2\leq i\leq 7$ with relations in degrees $\leq 12$. Compare this with the underlying tame curve, which by \cite[Ex.~9.2.3]{vzb} has generators in degrees $\leq 21$ with relations in degrees $\leq 42$. 
\end{ex}

\section{Ethereal modular forms}
\label{sec:etherealMF}
Fix an integer $N\geq 1$ and a prime $p$ that does not divide $N$. Let $\X = \X_{0}(N)$ be the stacky modular curve parametrizing elliptic curves with level $\Gamma_{0}(N)$ structure over $\Z\left [\frac{1}{N}\right ]$ and let $\X_{\F_{p}}$ be its fiber over $\F_{p}$. Also let $\omega = \Omega_{\X/\Z\left [\frac{1}{N}\right ]}(\Delta)$ be the log canonical sheaf on $\X$ so that by \cite[Lem.~6.2.3]{vzb}, $H^{0}(\X,\omega^{\otimes k/2}) \cong M_{k}\left (N;\Z\left [\frac{1}{N}\right ]\right )$, the ring of weight $k$, level $N$ Katz modular forms. Similarly, put $\omega_{\F_{p}} = \Omega_{\X/\F_{p}}(\Delta)$ so that $H^{0}(\X_{\F_{p}},\omega_{\F_{p}}^{\otimes k/2}) \cong M_{k}(N;\F_{p})$. There is an exact sequence 
$$
0\rightarrow H^{0}(\X,\omega^{\otimes k/2})\xrightarrow{p} H^{0}(\X,\omega^{\otimes k/2})\xrightarrow{r_{p}} H^{0}(\X,\omega_{\F_{p}}^{\otimes k/2})\rightarrow H^{1}(\X,\omega^{\otimes k/2})[p]. 
$$
The map $r_{p}$ is reduction mod $p$ and we call the complement of $\im(r_{p})$ in $M_{k}(N;\F_{p})$ the space of weight $k$, level $N$ {\it ethereal modular forms} mod $p$. 

\begin{rem}
In \cite{sch}, Schaeffer considers the analogous situation for the modular curves $X = X_{1}(N)$, which are honest algebraic curves for $p > 3$ or $N > 3$, and deduces facts about ethereal modular forms with level $\Gamma_{1}(N)$ structure for $N > 3$. The crucial difference with level $\Gamma_{0}(N)$ is that the curves $\X_{0}(N)^{\rig}$ are stacky for infinitely many $N\geq 1$. As a result, when $p = 2,3$, Riemann--Roch does not force $H^{1}(\X_{0}(N)^{\rig},\omega^{\otimes k/2})[p] = 0$, so $r_{p}$ is not surjective in general (see Theorems~\ref{thm:levels} and~\ref{thm:levelforms}). Other notable differences between the two cases include: 
\begin{itemize}
    \item Ethereal forms for $\Gamma_{1}(N)$ always have weight $1$ \cite[Thm.~1.7.1]{kat}, while for $\Gamma_{0}(N)$, ethereal forms may occur in other weights, as shown in several examples below. 
    \item For $\Gamma_{1}(N)$, ethereal forms are always cusp forms \cite[Prop.~8.3.1]{sch} and therefore their corresponding Galois representations are irreducible. In contrast, ethereal forms for $\Gamma_{0}(N)$ need not be cusp forms; see Example~\ref{ex:hasse}. 
    \item Ethereal forms for $\Gamma_{1}(N)$ appear to be sporadic and do not occur for small levels and characteristics; see \cite[App.~A]{sch}. In contrast, the wild stacky structure of $\X(1)$ predicts ethereal forms already in level $1$ and characteristics $2$ and $3$, and this phenomenon propagates up the tower of modular stacky curves $\X_{0}(N)\rightarrow\X(1)$ (again, see Theorems~\ref{thm:levels} and~\ref{thm:levelforms}). 
\end{itemize}
\end{rem}

\begin{ex}
\label{ex:hasse}
The original example of an ethereal modular form is the Hasse invariant $a_{p}$ in characteristics $p = 2,3$. Its $q$-expansion is $a_{p}(q) = 1 = 1 + 0q + 0q^{2} + \ldots$ Famously, for each prime $p$, $a_{p}$ is the mod $p$ reduction of the normalized Eisenstein series $e_{p - 1}$; when $p > 3$, $e_{p - 1}$ lies in $M_{p - 1}(1;\Z)$ and so $a_{p}$ is not ethereal. However, for $p = 2,3$, $e_{p - 1}$ is not a modular form, so $a_{p}$ is ethereal. 

On the other hand, for $p = 3$ and $\ell\equiv 1,11\pmod{12}$, pulling back $a_{3}$ along the covering map $\X_{0}(\ell)\rightarrow\X(1)$ determines a modular form $a_{3}(\ell)\in\M_{p - 1}(\ell;\F_{3})$ which is not ethereal, i.e.~$a_{3}(\ell)$ lifts to some modular form in $M_{2}\left (\ell;\Z\left [\frac{1}{\ell}\right ]\right )$. Explicitly, there is a unique weight $p - 1 = 2$ Eisenstein series $e_{2}(\ell)\in\M_{2}\left (\ell;\Z\left [\frac{1}{\ell}\right ]\right )$ such that $e_{2}(\ell) \equiv a_{2}(\ell) \pmod{2}$. The same thing happens for $p = 2$, $\ell\equiv 1,11\pmod{12}$ and the square of the Hasse invariant, $a_{2}(\ell)^{2}\in\M_{2}(\ell;\F_{2})$. This behavior appears to be specific to the Hasse invariant; see Example~\ref{ex:ethereallevel65mod2} and Question~\ref{q:ethereallift}. 
\end{ex}

\begin{ex}(\repolink{6.3.mod-forms.m} in \cite{KobinZB:magma-scripts-modular-forms-mod-p})
\label{ex:ethereallevel5mod2}
For $p = 2$ and $N = 5$, there are new generators in low degree by Theorem~\ref{thm:levels}. Let's analyze the ring $\M_{\bullet}(5;\F_{2})$ in more detail. From Example~\ref{ex:MFmod2level5}, this is a polynomial ring in two generators $x_{2},y_{2}$, each of weight $2$, with $x_{1} = 1$ (the lift of the Hasse invariant) and $y_{2}$ ethereal. In weight $4$, these give a basis consisting of $3$ forms, $x_{2}^{2},x_{2}y_{2},y_{2}^{2}$, each of which is the mod $2$ reduction of a classical form in weight $4$. On the other hand, {\magma} produces the following basis for the image of $\M_{4}\left (5;\Z\left [\frac{1}{5}\right ]\right )$ in $\M_{4}(5;\F_{2})$: 
$$
\{f_{1},f_{2},f_{3}\} = \{1,q + q^{5} + q^{9} + O(q^{25}),q^{2} + q^{4} + q^{8} + q^{10} + q^{16} + q^{18} + q^{20} + O(q^{32})\}. 
$$
We obviously have $f_{1} = x_{2}^{2}$ and we may take $f_{3} = y_{2}^{2}$, so that 
$$
y_{2} = q + q^{2} + q^{4} + q^{5} + q^{8} + q^{9} + q^{10} + q^{16} + q^{18} + q^{20} + O(q^{25}). 
$$
{\magma} also shows that $f_{2} = y_{2}^{2} + x_{2}y_{2}$. Since $x_{2} = 1$, this means $y_{2}$ is an Artin--Schreier root of the modular form $f_{2}$, which is the mod $2$ reduction of a classical cuspidal newform (see \href{https://www.lmfdb.org/ModularForm/GL2/Q/holomorphic/5/4/a/a/}{its page at the LMFDB}). There is also an ethereal form in each weight $k\equiv 2\pmod{4}$ given by $y_{2}^{k/2}$. 
\end{ex}

\begin{ex}(\repolink{6.4.mod-forms.m} in \cite{KobinZB:magma-scripts-modular-forms-mod-p})
\label{ex:ethereallevel7mod3}
We saw in Example~\ref{ex:MFmod3level7} that $\M_{\bullet}(7;\F_{3})$ is generated by $x_{2},y_{2}$ in weight $2$, with $y_{2}$ ethereal, and $x_{6}$ in weight $6$. In this case, $y_{2} = a_{3}(7)$ is still ethereal while $y_{2}^{2}$ lifts to a classical form of weight $4$. {\magma} produces a $q$-expansion for $x_{2}\in\M_{2}(7;\F_{3})$, namely  
$$
x_{2} = 1 + q + q^3 + q^4 + q^7 + q^9 + q^{12} + 2q^{13} + q^{16} + 2q^{19} + O(q^{21})
$$
as well as a basis $\{f_{1},f_{2},f_{3}\}$ for the image of the mod $3$ reduction map in $\M_{4}(7;\F_{3})$, with $f_{1} = y_{2}^{2}$, 
\begin{align*}
    f_{2} &= q + q^3 + q^4 + q^7 + q^9 + q^{12} + 2q^{13} + q^{16} + 2q^{19} + O(q^{21})\\
    \text{and}\quad f_{3} &= q^2 + 2q^4 + 2q^5 + q^6 + 2q^7 + q^{10} + 2q^{11} + 2q^{12} + q^{13} + 2q^{14} + 2q^{15} + 2q^{16}\\
        &\quad + q^{18} + q^{19} + q^{20} + O(q^{21}). 
\end{align*}
One can check that $f_{2} = x_{2}y_{2} - y_{2}^{2}$ and $f_{3} = x_{2}^{2} + x_{2}y_{2} + y_{2}^{2}$. {\magma} also produces a basis $\{1,h_{1},h_{2},h_{3},h_{4}\}$ for the entire space $\M_{6}(7;\F_{3})$, where 
\begin{align*}
    h_{1} &= q + q^5 + q^6 + q^8 + q^{10} + 2q^{12} + q^{14} + 2q^{15} + q^{17} + q^{18} + O(q^{21})\\
    h_{2} &= q^2 + q^5 + 2q^8 + 2q^{11} + q^{14} + 2q^{17} + q^{20} + O(q^{23})\\
    h_{3} &= q^3 + q^9 + q^{12} + O(q^{21})\\
    \text{and}\quad h_{4} &= q^4 + 2q^5 + 2q^6 + q^7 + 2q^8 + 2q^{10} + q^{12} + 2q^{13} + 2q^{14} + q^{15} + q^{16} + 2q^{17}\\
    &\quad + 2q^{18} + 2q^{19} + O(q^{21}). 
\end{align*}
This is related to the monomial basis from Example~\ref{ex:MFmod3level7} by 
$$
\{1,h_{1},h_{2},h_{3},h_{4}\} = \{y_{2}^{3},2x_{2}^{3} + x_{2}y_{2}^{2} + 2x_{6},x_{2}^{2} + x_{2}y_{2}^{2} + y_{2}^{3} + x_{6},x_{2}^{3} + 2y_{2}^{3},x_{6}\}. 
$$
Notice that the weight $2$ cusp form $g = x_{2} + 2y_{2}$ is an ethereal cube root of $h_{3}$. In higher weights, there is an ethereal form $x_{2}^{k/2} + 2y_{2}^{k/2}$ in each weight $k\equiv 2\pmod{6}$. 
\end{ex}

\begin{ex}(\repolink{6.5.mod-forms.m} in \cite{KobinZB:magma-scripts-modular-forms-mod-p})
\label{ex:ethereallevel13mod2}
For level $N = 13$ in characteristic $p = 2$, Example~\ref{ex:MFmod2level13} showed that $\M(13;\F_{2})$ is generated by $x_{2},y_{2},x_{4},y_{4},x_{6},y_{6}$, with subscripts indicating their weights. We can take $x_{2} = a_{2}(13)^{2}$, the square of the Hasse invariant (see Example~\ref{ex:hasse}). In weight $4$, {\magma} produces a basis $\{1,f_{1},f_{2},f_{3},f_{4}\}$ with 
\begin{align*}
    f_{1} &= q + q^9 + q^{13} + O(q^{25})\\
    f_{2} &= q^2 + q^5 + q^6 + q^7 + q^8 + q^{13} + q^{15} + q^{20} + O(q^{21})\\
    f_{3} &= q^3 + q^5 + q^6 + q^7 + q^9 + q^{10} + q^{12} + q^{14} + q^{15} + q^{17} + q^{18} + q^{20} + O(q^{21})\\
    \text{and}\quad f_{4} &= q^4 + q^5 + q^6 + q^7 + q^{13} + q^{15} + q^{16} + q^{18} + q^{20} + O(q^{21}). 
\end{align*}
One can check that $f_{2} + f_{4}$ is a square, so we can take $f_{2} + f_{4} = y_{2}^{2}$ for 
$$
y_{2} = q + q^2 + q^4 + q^8 + q^9 + q^{13} + q^{16} + q^{18} + O(q^{25}). 
$$
One can also check that $f_{1} + f_{2} + f_{4} = x_{2}y_{2}$, which can be rewritten $f_{1} = y_{2}^{2} + x_{2}y_{2}$, showing $y_{2}$ is an Artin--Schreier root of $f_{1}$. So the basis provided by {\magma} is really of the form $\{x_{2}^{2},x_{2}y_{2} + y_{2}^{2},y_{2}^{2} + y_{4},x_{4},y_{4}\}$. In weight $6$, {\magma} provides a basis $\{1,h_{1},h_{2},h_{3},h_{4},h_{5},h_{6}\}$, where 
\begin{align*}
    h_{1} &= q + q^9 + q^{13} + O(q^{25})\\
    h_{2} &= q^2 + q^8 + q^{10} + q^{12} + q^{14} + q^{18} + O(q^{30})\\
    h_{3} &= q^3 + q^9 + q^{13} + q^{17} + O(q^{25})\\
    h_{4} &= q^4 + q^{10} + q^{12} + q^{14} + q^{16} + O(q^{26})\\
    h_{5} &= q^5 + q^7 + q^{13} + q^{15} + O(q^{21})\\
    h_{6} &= q^6 + q^{10} + q^{12} + q^{14} + q^{18} + q^{20} + O(q^{24}). 
\end{align*}
These are related to the monomial basis in Example~\ref{ex:MFmod2level13} by 
$$
\{1,h_{1},h_{2},h_{3},h_{4},h_{5},h_{6}\} = \{x_{2}^{3},x_{2}^{2}y_{2} + x_{2}y_{2}^{2},x_{2}y_{2}^{2} + x_{2}y_{4} + x_{6} + y_{6},x_{2}x_{4} + x_{6} + y_{6},x_{2}y_{4} + x_{6} + y_{6},x_{6},y_{6}\}. 
$$
We also find the weight $6$ relations $y_{2}x_{4} = x_{2}y_{4}$ and $y_{2}y_{4} = y_{2}^{3} + x_{2}x_{4} + x_{2}y_{4}$, whereas there are no relations in the space of classical forms $\M_{6}\left (13;\Z\left [\frac{1}{13}\right ]\right )$. Relations in weights $8,10$ and $12$ can be found similarly, e.g.~$x_{4}^{2} = x_{2}y_{6}$. 
\end{ex}

\begin{ex}(\repolink{6.6.mod-forms.m} in \cite{KobinZB:magma-scripts-modular-forms-mod-p})
\label{ex:ethereallevel13mod3}
The graded ring $\M(13;\F_{3})$ is also interesting: by Example~\ref{ex:MFmod3level13}, it is generated by $x_{2},y_{2},x_{4},y_{4},x_{6}$, with $y_{2}$ ethereal. In this case, $y_{2} = a_{3}(13)$ and $y_{2}^{2}$ is non-ethereal. The basis for $\M_{4}(13;\F_{3})$ produced by {\magma} is $\{1,f_{1},f_{2},f_{3},f_{4}\}$ with 
\begin{align*}
    f_{1} &= q + q^5 + 2q^6 + q^7 + q^9 + 2q^{10} + q^{11} + q^{12} + 2q^{15} + q^{16} + 2q^{17} + O(q^{21})\\
    f_{2} &= q^2 + q^6 + 2q^7 + q^{10} + 2q^{11} + 2q^{13} + 2q^{14} + q^{16} + q^{17} + q^{18} + 2q^{19} + q^{20} + O(q^{21})\\
    f_{3} &= q^3 + q^5 + q^6 + q^7 + 2q^8 + q^{11} + 2q^{12} + q^{14} + q^{16} + q^{17} + O(q^{24})\\
    f_{4} &= q^4 + q^5 + q^8 + q^{10} + q^{11} + q^{12} + q^{13} + 2q^{14} + q^{15} + 2q^{16} + 2q^{19} + O(q^{23})
\end{align*}
which can be written in the monomial basis as 
$$
\{1,f_{1},f_{2},f_{3},f_{4}\} = \{y_{2}^{2},x_{4},x_{2}^{2} + x_{2}y_{2} + y_{2}^{2} + y_{4},2x_{2}y_{2} + y_{2}^{2} + 2x_{4} + 2y_{4},y_{4}\}. 
$$
In weight $6$, the images of {\magma}'s basis of modular forms can be written in the monomial basis as 
\begin{align*}
    & \{y_{2}^{3},x_{6},x_{2}^{3} + x_{2}^{2}y_{2} + y_{2}^{3} + y_{2}x_{4} + x_{6},2x_{2}^{3} + y_{2}^{3},x_{2}^{3} + 2x_{2}y_{2}^{2} + 2x_{6},\\
    &\qquad 2x_{2}^{3} + x_{2}^{2}y_{2} + x_{2}y_{2}^{2} + y_{2}^{3} + x_{2}x_{4} + 2y_{2}x_{4} + x_{6},\\
    &\qquad 2x_{2}^{3} + x_{2}^{2}y_{2} + x_{2}y_{2}^{2} + y_{2}^{3} + x_{2}x_{4} + y_{2}x_{4} + 2x_{6}\}. 
\end{align*}
We also have the weight $6$ relations $x_{2}y_{4} = 2x_{2}^{3} + 2x_{2}^{2}y_{2} + 2x_{2}y_{2}^{2} + x_{2}x_{4}$ and $y_{2}y_{4} = 2x_{2}^{2}y_{2} + 2x_{2}y_{2}^{2} + 2y_{2}^{3} + y_{2}x_{4}$ which do not appear in $\M_{6}\left (13;\Z\left [\frac{1}{13}\right ]\right )$. There are further relations in weights $8$ and $10$. Finally, note that the form $h = x_{2} + 2y_{2}$ is an ethereal cusp form in $\M_{2}(13;\F_{3})$ and for each $k\equiv 2\pmod{6}$, $x_{2}^{k/2} + 2y_{2}^{k/2}$ is ethereal of weight $k$. 
\end{ex}

One thing to notice in the examples above is that the Hasse invariant does not always lift along $\X_{0}(\ell)\rightarrow\X_{0}(1)$ to the image of a classical modular form; that is, sometimes the Hasse invariant remains ethereal. A natural question to ask is: 
\begin{question}
\label{q:ethereallift}
Is there an ethereal form $f\in\M(N;\F_{p})$, for $p = 2$ or $3$ and for some $N > 1$, such that $\pi^{*}f$ remains ethereal in $\M(MN;\F_{p})$ for infinitely many $M > 1$, where $\pi \colon \X_{0}(MN)\rightarrow\X_{0}(N)$ is the natural projection? 
\end{question}

\begin{ex}(\repolink{6.8.mod-forms.m} in \cite{KobinZB:magma-scripts-modular-forms-mod-p})
\label{ex:ethereallevel65mod2}
The ring $\M(65;\F_{2})$ is generated by $10$ forms in weight $2$, with $31$ relations in weight $4$. Let the $10$ generators in $\M_{2}(65;\F_{2})$ be labeled\footnote{To keep notation simple in this example, subscripts no longer indicate weight.} $x_{1},\ldots,x_{10}$. Since $\dim\M_{2}\left (65;\Z\left [\frac{1}{65}\right ]\right ) = 8$, two of these generators are ethereal, say $x_{9}$ and $x_{10}$. 

In weight $4$, we have $\dim\M_{4}\left (65;\Z\left [\frac{1}{65}\right ]\right ) = \dim\M_{4}(65;\F_{2}) = 24$. One of the $24$ modular forms in a basis for $\M_{4}\left (65;\Z\left [\frac{1}{65}\right ]\right )$ reduces mod $2$ to a modular form which has a square root with $q$-expansion 
$$
x_{9} = q^{13} + q^{26} + q^{52} + q^{65} + O(q^{104}). 
$$
(One can check that this is not in the span of $x_{1},\ldots,x_{8}$, so it is ethereal.) This is visibly an oldform coming from level $5$, namely the pullback along $\X_{0}(65)^{\rig}\rightarrow\X_{0}(5)^{\rig}$ of the ethereal weight $2$ form $y_{2}\in\M_{2}(5;\F_{2})$ that we found in Example~\ref{ex:ethereallevel5mod2}. Therefore $y_{2}$ remains ethereal in $\M_{2}(65;\F_{2})$. Likewise, one can check that the weight $2$ ethereal form of level $13$ from Example~\ref{ex:ethereallevel13mod2} remains ethereal in $\M_{2}(65;\F_{2})$. 

Further, $x_{9}$ is an Artin--Schreier root of the mod $2$ reduction of another form in this basis: $x_{9}^{2} + x_{1}x_{9} = g$ where $x_{1}$ is the Hasse invariant and 
$$
g = q^{13} + q^{65} + O(q^{117}). 
$$
On the other hand, we have the relation $y_{2} = x_{2} + x_{3} + x_{5} + x_{6} + x_{9}$. This can be viewed as an example of level-lowering for ethereal forms: the level $65$ ethereal form $x_{2} + x_{3} + x_{5} + x_{6} + x_{9}$ is actually an ethereal form of level $5$. 

There are two other forms in the basis for $\M_{4}\left (65;\Z\left [\frac{1}{65}\right ]\right )$ whose reductions mod $2$ are 
\begin{align*}
    h_{1} &= q^8 + q^{23} + q^{28} + q^{31} + q^{33} + q^{36} + q^{46} + O(q^{53})\\
    \text{and}\quad h_{2} &= q^{12} + q^{23} + q^{24} + q^{30} + q^{31} + q^{33} + q^{34} + q^{38} + q^{40} + q^{42} + O(q^{50})
\end{align*}
and whose sum $h_{1} + h_{2}$ is a square: $h_{1} + h_{2} = x_{10}^{2}$ with 
$$
x_{10} = q^4 + q^6 + q^{12} + q^{14} + q^{15} + q^{17} + q^{18} + q^{19} + q^{20} + q^{21} + O(q^{23}). 
$$
One can also check that $x_{10}$ does not lie in the oldspace of forms coming from levels $1,5$ or $13$. This gives us a full set of generators $\{x_{1},\ldots,x_{8},x_{9},x_{10}\}$ for $\M_{\bullet}(65;\F_{2})$ and one can use linear algebra to express the mod $2$ reductions of a basis for $\M_{k}\left (65;\Z\left [\frac{1}{65}\right ]\right )$, for any $k\geq 2$, in terms of monomials in the $x_{i}$. 
\end{ex}

\begin{ex}(\repolink{6.9.mod-forms.m} in \cite{KobinZB:magma-scripts-modular-forms-mod-p})
\label{ex:ethereallevel91mod3}
Similarly, $\M_{2}(91;\F_{3})$ is generated by $12$ forms, say $x_{1},\ldots,x_{12}$, with $x_{11}$ and $x_{12}$ ethereal. One particular choice of $x_{11}$ and $x_{12}$ has mod $3$ $q$-expansions 
\begin{align*}
    x_{11} &= q^{7} + q^{21} + O(q^{28})\\
    \text{and}\quad x_{12} &= q + 2q^{2} + q^{5} + 2q^{17} + 2q^{18} + 2q^{19} + 2q^{21} + q^{22} + 2q^{23} + 2q^{24} + O(q^{27}). 
\end{align*}
The other generators $x_{1},\ldots,x_{10}$ can be found by reducing a basis for $\M_{2}\left (91;\Z\left [\frac{1}{91}\right ]\right )$ mod $3$ in {\magma}. After computing more terms in the $q$-expansion of $x_{11}$, it is clear that this form is an oldform from level $13$, namely $x_{11}(q) = h(q^{7})$ for the ethereal cusp form $h\in\M_{2}(13;\F_{3})$ identified in Example~\ref{ex:ethereallevel13mod3}. 

In this basis, the weight $2$ ethereal cusp form $g\in\M_{2}(7;\F_{3})$ from Example~\ref{ex:ethereallevel7mod3} can be written $g = x_{2} + x_{4} + x_{5} + 2x_{8} + x_{10} + 2x_{11}$. Likewise, the ethereal cusp form $h$ in $\M_{2}(13;\F_{3})$ satisfies $h = x_{2} + x_{4} + x_{5} + x_{8} + x_{10} + x_{11}$. As in Example~\ref{ex:ethereallevel65mod2}, this can be seen as a level-lowering phenomenon for these ethereal forms. There are no such relations involving $x_{12}$, so this is a genuinely ``new'' ethereal form of level $91$. 
\end{ex}

\subsection{Ethereal Galois representations}
\label{sec:ethrep}

The work of Deligne--Serre \cite{deligne-serre}, Khare--Wintenberger \cite{kw1,kw2} and Kisin \cite{kis} establishes a bijective correspondence 
\begin{equation}\label{eq:serre}
\{\text{representations } \rho \colon G_{\Q} \rightarrow \GL_{2}(\overline{\F}_{p})\} \longleftrightarrow \{\text{mod $p$ eigenforms } f\in\M_{k}(N;\overline{\F}_{p}) \mid N\geq 1\}
\end{equation}
which sends irreducible representations to cusp forms and identifies suitable notions of weight and level on each side. In particular, given a cuspidal newform $f = \sum_{n\geq 0} a_{n}q^{n}$ of level $N$, the attached representation $\rho = \rho_{f}$ is unramified away from $Np$ and 
\begin{equation}\label{eq:tr-det-newform}
\Tr(\rho(\Frob_{\ell})) = a_{\ell} \quad\text{and}\quad \det(\rho(\Frob_{\ell})) = \ell^{k - 1}
\end{equation}
for all $\ell\nmid pN$. From these properties, it may be possible to deduce the representation (up to semisimplification) in some examples. For a Galois representation $\rho$, denote by $\rho^{\operatorname{ss}}$ its semisimplification. Let us first recall a non-ethereal example. 

\begin{ex}
\label{ex:deltamod3}
Let $p = 3$ and let $f = \delta\in\S_{12}(1;\F_{3})$ be the unique cusp form of weight $12$ and level $1$. Its mod $3$ $q$-expansion is 
$$
\delta(q) = q + q^{4} + 2q^{7} + 2q^{13} + q^{16} + 2q^{19} + \ldots
$$
This is the mod $3$ reduction of the usual cusp form in $\S_{12}(1;\Z)$. Let $\rho = \rho_{\delta} \colon G_{\Q}\to\GL_{2}(\F_{3})$ be the representation attached to $\delta$, which is the mod $3$ reduction of the representation $G_{\Q}\rightarrow \SL_{2}(\Z)$ coming from Ramanujan's $\tau$-function. We know by (\ref{eq:tr-det-newform}) that for all primes $\ell\not = 3$, $\rho(\Frob_{\ell})$ has characteristic polynomial 
$$
\ch(\rho(\Frob_{\ell})) = t^{2} - \tau(\ell)t + \ell^{11} \equiv t^{2} - \tau(\ell)t + \ell \pmod{3}. 
$$
Note that for all primes $\ell\not = 3$, $\det(\rho(\Frob_{\ell})) = \ell = \chi(\Frob_{\ell})$, where $\chi$ is the mod $3$ cyclotomic character, and $\tau(\ell) \equiv 1 + \ell \pmod{3}$ by Ramanujan, so one might guess that $\rho = 1\oplus\chi$. Indeed this is true, as shown by Serre (using an argument of Tate) in \cite[note 229.2]{ser-note}. In this case, the image of $\rho_{\delta,3}^{\operatorname{ss}}$ is $(\Z/3\Z)^{\times}$, or more precisely
$$
\im(\rho_{\delta,3}^{\operatorname{ss}}) = \begin{pmatrix} 1 & 0 \\ 0 & \pm 1\end{pmatrix}\subseteq\GL_{2}(\F_{3}),
$$
and the corresponding number field $K_{\delta}$ with $\Gal(K_{\delta}/\Q) \cong \im(\rho_{\delta,3}^{\operatorname{ss}})$ must be $\Q\left (\sqrt{-3}\right )$. Note that $\rho_{\delta,3}$ is reducible even though $\delta$ is a cusp form; this is because $\delta\equiv e_{12} - 1\pmod{3}$, where $e_{12}$ is the weight $12$ Eisenstein series. 
\end{ex}

\begin{ex}
For $p = 2$, a similar calculation as in Example~\ref{ex:deltamod3} shows that $\rho_{\delta,2}^{\operatorname{ss}} \cong 1\oplus 1$, again using \cite[note 229.2]{ser-note}. In particular, $K_{2} = \Q$. 
\end{ex}

%

\begin{ex}
For $p = 2,3$, the Hasse invariant $a_{p}$ also corresponds to the trivial Galois representation. Of course, this is no surprise as $a_{p}$ is the mod $p$ reduction of an Eisenstein series (see Example~\ref{ex:hasse}). 
\end{ex}

In the correspondence (\ref{eq:serre}), irreducible representations are sent to cusp forms, but the converse is not necessarily true. This means that some of our interesting examples of ethereal modular forms may fail to determine ethereal Galois representations with ``large image'', in the sense of \cite{buz}. Below, we work through a few examples of such forms in weight $2$, ending with several questions for future work. 

\begin{ex}(\repolink{6.13.mod-forms.m} in \cite{KobinZB:magma-scripts-modular-forms-mod-p})
Let $y_{2}$ be the ethereal weight $2$ form in $M_{2}(5;\F_{2})$ from Example~\ref{ex:ethereallevel5mod2}. Write $y_{2} = \sum_{n\geq 1} a_{n}q^{n}$ and let $\rho \colon G_{\Q}\to\GL_{2}(\F_{2})$ be the corresponding Galois representation. Analyzing the $q$-expansion in {\magma}, it appears that $a_{\ell} = 0$ for all primes $\ell\not = 2,5$, so $\im(\rho^{\operatorname{ss}})$ likely consists of trace $0$ matrices\footnote{By the Chebotarev density theorem, $\rho(\Frob_{2})$ and $\rho(\Frob_{5})$ must have trace $0$ as well.} and therefore must be trivial or one of the three order $2$ subgroups of $\GL_{2}(\F_{2})$. In particular, $\rho$ appears to be reducible. 
\end{ex}

\begin{ex}(\repolink{6.14.mod-forms.m} in \cite{KobinZB:magma-scripts-modular-forms-mod-p})
Next, consider the ethereal form $y_{2} = \sum_{n\geq 1} a_{n}q^{n}\in M_{2}(13;\F_{2})$ from Example~\ref{ex:ethereallevel13mod2} and let $\rho \colon G_{\Q}\to\GL_{2}(\F_{2})$ be the corresponding Galois representation. Similar to the previous example, computations show that $a_{\ell} = 0$ for all primes $\ell\not = 2,13$, so $\im(\rho^{\operatorname{ss}})$ has order $\leq 2$ and $\rho$ is again reducible. 
\end{ex}

\begin{ex}(\repolink{6.15.mod-forms.m} in \cite{KobinZB:magma-scripts-modular-forms-mod-p})
Let $x_{1},\ldots,x_{10}$ be the basis for $M_{2}(65;\F_{2})$ from Example~\ref{ex:ethereallevel65mod2}, with $x_{9}$ and $x_{10}$ ethereal. Computations suggest that for $\ell\not\in\{2,5,13\}$, the Hecke operators $T_{\ell}$ acts nilpotently on the space of cusp forms, with a basis of $4$ normalized Hecke eigenforms: 
$$
\{f_{1},f_{2},f_{3},f_{4}\} = \{x_{2} + x_{3} + x_{5} + x_{6},x_{2} + x_{3} + x_{5} + x_{9},x_{2} + x_{5} + x_{6} + x_{9},x_{2} + x_{5} + x_{9}\}. 
$$
Notice that three of these are ethereal and only $f_{2}$ lies in the oldspace of $S_{2}(65;\F_{2})$ -- see Example~\ref{ex:ethereallevel65mod2}. Therefore $f_{3}$ and $f_{4}$ determine ``new'' ethereal irreducible Galois representations $\rho_{f_{3}},\rho_{f_{4}} \colon G_{\Q}\to\GL_{2}(\F_{2})$. 

The non-ethereal newform $f_{1}$ is the mod $2$ reduction of the form labeled \href{https://www.lmfdb.org/ModularForm/GL2/Q/holomorphic/65/2/a/a/}{65.2.a.a} in the LMFDB, which corresponds to the elliptic curve with Cremona label \href{https://www.lmfdb.org/EllipticCurve/Q/65a1/}{65a1}; that is, this curve, defined by the Weierstrass equation $y^{2} + xy = x^{3} - x$, has mod $2$ Galois representation $\bar{\rho}_{E,2} = \rho_{f_{1}}$. 

In characteristic $0$, there are also two irrational newforms in $S_{2}(65;\overline{\Q})$ defined over $\Q\left (\sqrt{3}\right )$ and $\Q\left (\sqrt{2}\right )$, respectively. Their mod $2$ reductions are both equal to $f_{1}$, which confirms that $f_{3}$ and $f_{4}$ are genuine ethereal forms. Nevertheless, the images of $\rho_{f_{3}}^{\operatorname{ss}}$ and $\rho_{f_{4}}^{\operatorname{ss}}$ once again appear to consist of trace $0$ matrices in $\GL_{2}(\F_{2})$, so the representations are likely reducible, as in the previous examples. 

On the other hand, since the Hecke algebra acts nilpotently, this implies that any normalized cusp form outside the span of $\{f_{1},f_{2},f_{3},f_{4}\}$ -- for example, any form with nonzero $x_{10}$ component -- is not an eigenform, but a generalized eigenform. By the deformation theory of Mazur \cite{maz2}, each such form determines an extension class in $\Ext^{1}(\rho_{f_{i}},\rho_{f_{j}})$, where for $1\leq i\leq 4$, $\rho_{f_{i}}$ is the Galois representation attached to the eigenform $f_{i}$. It would be interesting to compute such a deformation explicitly. 
\end{ex}

\begin{ex}(\repolink{6.16.ethereal-reps.m} in \cite{KobinZB:magma-scripts-modular-forms-mod-p})
Let $g = x_{2} + 2y_{2}$ be the ethereal cusp form in $\M_{2}(7;\F_{3})$ from Example~\ref{ex:ethereallevel7mod3} and suppose $\rho = \rho_{g} \colon G_{\Q}\to\GL_{2}(\F_{3})$ is the corresponding irreducible Galois representation. One can check that $g$ is new. Computations show that for roughly 50\% of primes $\ell < 500$, $a_{\ell}(g) = 0$ and for roughly 50\%, $a_{\ell}(g) = 2$, so the image of $\rho^{\operatorname{ss}}$ lies in a subgroup of $\GL_{2}(\F_{3})$ isomorphic to $S_{3}$. In particular, $\rho$ is likely reducible with $\rho^{\operatorname{ss}} \cong 1\oplus\chi$, where $\chi$ is the cyclotomic character. 
\end{ex}

\begin{question}
\label{q:etherealeigenform}
For $p = 2$ or $3$, are there infinitely many levels $N$ for which there exists an ethereal Hecke eigenform\footnote{There are clearly ethereal eigenforms at many levels by the examples above, but most seem to be congruent to an Eisenstein series (up to constant terms). However, we can require that $f$ also be an eigenform for $T_{\ell}$ with $\ell\mid N$, which makes the question interesting and far less obvious.} $f\in\M_{2}(N;\F_{p})$? 
\end{question}

To get a better handle on newforms in particular, it would be desirable to have a version of Theorem~\ref{thm:X0N-intro} for the Shimura curves $\X_{0}^{D}(N)$.

\section{Nonstandard level structures}
\label{sec:data}
The strategy outlined in \SS\ref{subsec:algorithm} for computing rings of mod $p$ modular forms from stacky modular curves works well for any finite-index subgroup $H\subset\SL_{2}(\Z)$. We have given a thorough treatment of the $\Gamma_{0}(N)$ case and will leave other families of subgroups for future work (see Section~\ref{sec:future}), but in this section, let us examine one such group in detail. 

\begin{ex}(\repolink{7.1.nonsplit-cartan-level-3.m} in \cite{KobinZB:magma-scripts-modular-forms-mod-p})
\label{ex:nonsplitcartanlevel3}
Let $H = N_{\ns}(3)$ be the normalizer of the non-split Cartan subgroup of level $3$. The corresponding modular curve $\X_{\ns}^{+}(3)$ (here is its \href{https://beta.lmfdb.org/ModularCurve/Q/3.3.0.a.1/}{page on the LMFDB beta}) is a stacky $\P^{1}$ whose rigidification $\X_{\ns}^{+}(3)^{\rig}$ has three $\mu_{2}$-points and one cusp. Therefore the log canonical ring $R(\X_{\ns}^{+}(3)^{\rig};\Q) \cong M_{\bullet}(N_{\ns}(3);\Q)$ admits a presentation of the form\footnote{It turns out that $x,y$ and $z$ can be chosen such that $f = x^{3} + y^{3} - z^{2}$, providing a fascinating link between $\X_{\ns}^{+}(3)$ and the generalized Fermat equation of signature $(2,3,3)$.} $\Q[x,y,z]/(f)$ with $\deg x = \deg y = 2$ and $\deg z = 4$ (corresponding to modular forms of weights $4,4$ and $6$, respectively) and $\deg f = 6$. 

In characteristic $2$, the stacky points collide and $\X_{\ns}^{+}(3)^{\rig}$ becomes a stacky $\P^{1}$ with a single stacky $V_{4}$-point\footnote{{\magma} shows the order of the automorphism group at this point is $4$. Since the point lies over $j = 0$ in $\X(1)$, the automorphism group must be a subgroup of $\Aut(E)/\{\pm 1\}$ where $E$ is the unique supersingular elliptic curve in characteristic $2$. By Proposition~\ref{prop:ellauts}, the only such order $4$ subgroup is isomorphic to $V_{4}$.}, say $P$, and one cusp. Thus the log canonical divisor is of the form 
$$
K_{\X_{\ns}^{+}(3)} + \Delta = -\infty + aP
$$
for some $a$. As long as\footnote{A priori, it is possible that $a = 0$, although we were able to show that this isn't the case here. However, if $a = 0$, then all modular forms in $M_{k}\left (N_{\ns}(3);\Z\left [\frac{1}{3}\right ]\right )$, for $k > 2$, would reduce to $0$ mod $2$. This seems unlikely, as there should be at least one normalized eigenform somewhere in the ring of modular forms for $N_{\ns}(3)$.} $a > 0$, there will be a weight $2$ modular form and since there are no weight $2$ forms in characteristic $0$, such a form must be ethereal. 

To compute $a$, we first consider the modular curve $\X_{\ns}(3)$ associated to the non-split Cartan subgroup $C_{\ns}(3)$. Note that in characteristic $0$, there is an \'{e}tale double cover of stacks $\X_{\ns}(3)\to\X_{\ns}^{+}(3)$ corresponding to the inclusion of subgroups $C_{\ns}(3)\subset N_{\ns}(3)$. By the moduli interpretation of these curves, the cover remains \'{e}tale in characteristic $2$, so we can lift along this cover to simplify the computation of $K_{\X_{\ns}^{+}(3)}$. 

Rigidifying, $\X_{\ns}(3)^{\rig}$ is a stacky $\P^{1}$ with a single stacky $\Z/2\Z$-point at $j = 0$. Since the coarse space is $\P^{1}$, \cite[Thm.~6.18]{kob} shows that over $\overline{\F}_{2}$, $\X_{\ns}(3)^{\rig}$ is an Artin--Schreier root stack of $\P^{1}$ at $j = 0$; in turn, [{\it loc.~cit.}, Ex.~6.12] provides a ramified $\Z/2\Z$-cover $Y\rightarrow\P^{1}$ such that $\X_{\ns}(3)^{\rig} \cong [Y/(\Z/2\Z)]$. 

In fact, we can choose such a cover over $\F_{2}$, with $Y \cong \P^{1}$. By the moduli interpretation of $\X_{\ns}(3)$ in \cite{rw}, there is a double cover $\P^{1} = X(3)\to X_{\ns}(3)\cong\P^{1}$ ramified over the point above $j = 0$ and defined over $\Z[1/3]$; this determines our \'{e}tale double cover of stacks $Y = X(3)\to\X_{\ns}(3)$. Now write $K_{\X_{\ns}(3)} = -2\infty + b\tilde{P}$ for some $b$. Then, as in \SS\ref{sec:tangent23}, we have
$$
-2 = \deg K_{\P^{1}} = 2\deg K_{\X_{\ns}(3)} = 2\left (-2 + \frac{b}{2}\right ) = b - 4
$$
which implies $b = 2$. Similarly, 
$$
-2 + \frac{b}{2} = \deg K_{\X_{\ns}(3)} = 2\deg K_{\X_{\ns}^{+}(3)} = 2\left (-2 + \frac{a}{4}\right ) = \frac{a}{2} - 4
$$
and we conclude that $a = 6$. This is consistent with a ramification filtration $G_{0} = G_{1} = V_{4}$ and $G_{i} = 0$ for $i > 1$, i.e.~with ramification jumps $m_{1} = m_{2} = 2$. 

From this, we obtain a presentation of the ring of mod $2$ modular forms of $N_{\ns}(3)$ of the form $\F_{2}[x,y]$ with $\deg x = 1$ and $\deg y = 2$ (so weights $2$ and $4$). In particular, $x$ is an ethereal form. One can use the results in \cite{ms} together with the techniques in this paper to obtain a $q$-expansion of $y$ and any other ethereal form in this ring. 
\end{ex}

\begin{rem}
Using the moduli interpretation of $\X_{\ns}^{+}(3)$, we could have directly analyzed the $V_{4}$-cover $X(3)\to\X_{\ns}^{+}(3)^{\rig}$. However, we included the two-step computation above because it yields the canonical divisor of $\X_{\ns}(3)$ as a bonus result. Additionally, the intermediate cover method is useful more generally for analyzing other modular curves (see Problem~\ref{prob:nonstandardlevels}). 
\end{rem}

\begin{rem}
The above is an example in which non-cyclic automorphism groups arise in the mod $p$ fibers of stacky modular curves. We saw in Example~\ref{ex:PSL2quotientstack} that the modular curve $X(1)$  has \emph{nonabelian} stabilizers
in characteristics $2$ and $3$.

In fact, any modular curve of the form $\X_H$ (see \cite[Section 2]{rszbv:ell-adic-images} for a definition) such that $H$ contains $\Aut E$ with $j(E) = 0$ in characteristic 2 or 3 has a stabilizer isomorphic to $\Aut E$. For example, one can check that the normalizer $N_{\spl}(11)$ of the split cartan of level 11 contains $\Aut E_{\F_2}$, so the stacky modular curve $\X^+_{\spl}(11)$ has a point with a nonabelian stabilizer in characteristic 2.

The maximal subgroups of $\GL_2(\F_{\ell})$ containing $\Aut E$ vary from prime to prime; for example, $N_{\spl}(7)$ does not contain a subgroup isomorphic to $\Aut E$ in characteristic 2 (the orders are incompatible), and the only maximal subgroup of $\GL_2(\F_{7})$ containing $\Aut E$ does not have surjective determinant (and thus does not correspond to a ``classical" modular curve).
    
\end{rem}

\begin{rem}
Whenever two stacky points of a modular curve $\X = \X_{H}$ collide mod $p$, one should expect a congruence of modular forms. For example, suppose that $\X$ is has good reduction mod $p$ and at least two cusps -- call one of them $C$ -- and suppose for simplicity that the genus of the coarse space of $\X$ is zero. Also, suppose $P$ and $Q$ are two stacky points which collide mod 2 and whose stabilizers both have order $e > 1$ and stay the same or increase in size when reducing mod $p$. 

With this setup, $K_{{\mathcal{X}}} + \Delta$ is effective (since there are two cusps) and includes $(e - 1)P$ and $(e - 1)Q$ in its support. Then there are functions $f$ and $g$ on $\mathbb{P}^1$ such that $\div f = P-C$ and $\div Q = Q-C$; these are unique up to scaling. These represent elements of $H^0({\mathcal{X}},2(K_{{\mathcal{X}}} + \Delta)$). But they are also elements of $H^0(P-C)$ (respectively $H^0(Q-C)$), which are 1 dimensional, spanned by $f$ and $g$. Mod $p$, $P = Q$, so the reductions $\overline{f}$ and $\overline{g}$ mod $p$ both lie in $H^0(P-C) = H^0(Q-C)$, which are still one dimensional. So $\overline{f}$ and $\overline{g}$ differ by a multiple, and we get a congruence. 
\end{rem}

It is trickier, however, to detect when modular forms that become congruent mod $p$ are replaced by linearly independent ethereal forms, as seen in the proof of Theorem~\ref{thm:levels} and in the example above. 


\section{Future directions}
\label{sec:future}
%
Our results on the stacky curves $\X_{0}(N)$ and mod $p$ modular forms of level $\Gamma_{0}(N)$ suggest several natural directions for generalization. 

\subsection{Nonstandard level structures}

As discussed in Section~\ref{sec:data}, our approach to the stacky curves $\X_{0}(N)$ is well-suited to studying modular curves coming from other subgroups $H\subset\SL_{2}(\Z)$. The beta version of the LMFDB offers many candidates in the \href{https://beta.lmfdb.org/ModularCurve/Q/}{modular curves section}, including $\X_{sp}^{+}(N),\X_{ns}^{+}(N),\X_{ns}(N)$ and $\X_{S_{4}}(\ell)$ for $\ell$ prime. 

\begin{problem}
\label{prob:nonstandardlevels}
Give a systematic treatment of each of these families of stacky curves in characteristic $p$. Find any ethereal modular forms and describe the resulting ethereal Galois representations. 
\end{problem}

We saw in \SS\ref{sec:ethrep} that in many examples, ethereal cusp forms of level $\Gamma_{0}(N)$ are congruent to an Eisenstein series (mod constant term) and hence give rise to reducible mod $p$ Galois representations. Frank Calegari suggested to the authors that this might always be the case for level $\Gamma_{0}(N)$ forms. However, this raises the following questions for future investigation. 

\begin{question}
Is there an ethereal cuspidal eigenform $f\in\M_{2}(H;\F_{2})$ for some $H\subset\SL_{2}(\Z)$ with $\im(\rho_{f}^{\operatorname{ss}}) = \GL_{2}(\F_{2})$? 
\end{question}

\begin{question}
Is there an ethereal cuspidal eigenform $f\in\M_{2}(H;\F_{3})$ for some $H\subset\SL_{2}(\Z)$ with $\SL_{2}(\F_{3})\subseteq\im(\rho_{f}^{\operatorname{ss}})\subseteq\GL_{2}(\F_{3})$? 
\end{question}

\subsection{Level divisible by the characteristic}
\label{sec:igusa}

To handle mod $p$ modular forms of level $N$ with $p\mid N$, more care is required. One approach would be to use Igusa curves $\Ig(p^{n})$, as in \cite[12.6.1]{km}, viewed as ramified covers of $X(1)$. One thing to note is that in the case $n = 1$, \cite[12.8.2]{km} shows that $\Ig(p)$ is isomorphic to the moduli problem of $(p - 1)$st roots of the Hasse invariant. It is possible there is a connection between $\Ig(p)$ and our root stack description of $\X(1)$ in \SS\ref{sec:localstructures}, though we did not attempt to realize this strategy in detail.

\subsection{Higher dimensional moduli problems}


In higher dimensions, one encounters many canonical rings of arithmetic interest. For example, for $g > 1$, let $\mathcal{A}_{g}$ be the moduli stack of principally polarized abelian varieties of dimension $g$, which is a smooth Deligne--Mumford stack of dimension $\frac{g(g + 1)}{2}$. Then, for an appropriate compactification $\overline{\mathcal{A}}_{g}$ of $\mathcal{A}_{g}$, there is an analogue of the Kodaira--Spencer map (\ref{eq:KS}) identifying the canonical ring of $\overline{\mathcal{A}}_{g}$ with the graded ring $\M_{\bullet}(\Sp_{2g}(\Z))$ of Siegel modular forms of dimension $g$. One similarly constructs rings of Siegel modular forms of level $N$ from the moduli stacks $\mathcal{A}_{g}(N)$ parametrizing principally polarized abelian varieties with level $N$ structure. 

Many results are known about the rings $\M_{\bullet}(\Sp_{2g}(\Z))$, due to Igusa, Tsuyumine, and others \cite{BvdGHZ}. We close with a pair of questions that we hope will inspire further interest in rings of Siegel modular forms, especially in characteristic $p$. Note that a similar line of inquiry is also proposed in the PhD thesis of Cerchia \cite{cerchia-thesis}. 

\begin{question}
For fixed $g > 1$, is there a uniform bound on the weights of generators and relations in a minimal presentation of the graded ring of Siegel modular forms of dimension $g$ and level $N$, where $p\nmid N$? 
\end{question}

\begin{question}
Do there exist a dimension $g > 1$, a prime $p$ and a level $N$ not divisible by $p$ for which there are ethereal mod $p$ Siegel modular forms of dimension $g$ and level $N$ in low weight? 
\end{question}


\end{document}